\documentclass[preprint,12pt]{elsarticle}\headheight 20mm \oddsidemargin 0mm \evensidemargin 0mm \topmargin
-20mm \textheight 220mm \textwidth 165mm

\usepackage{amsmath,amsthm,amssymb,graphics,algorithmicx,listings}

\usepackage{inputenc}
\usepackage{indentfirst}
\usepackage{graphicx}
\usepackage{float}
\usepackage{booktabs}
\usepackage{caption}
\usepackage{subcaption}
\usepackage{algorithm}
\usepackage{algpseudocode}
\usepackage{array}
\usepackage{color, xcolor}

\newtheorem{lemma}{Lemma}
\newtheorem{theorem}{Theorem}

\journal{XXX}

\begin{document}

\begin{frontmatter}



\title{Continuous Data Assimilation  for the Navier-Stokes Equations with Nonlinear Slip Boundary Conditions \tnoteref{t1}}

\tnotetext[t1]{This work is partially supported by the Natural Science Foundation of Chongqing
			(No. CSTB2024NSCQ-MSX0221).}
		
		\author{Wancheng Wu}
		\author{Haiyun Dong}
		\author{Kun Wang\corref{cor1}}
		\address{College of Mathematics and Statistics, Chongqing University, Chongqing 401331, P.R. China}
	\cortext[cor1]{Corresponding author.}

%

\begin{abstract}
    This paper focuses on continuous data assimilation (CDA) for the Navier-Stokes equations with nonlinear slip boundary conditions.  CDA methods are typically employed to recover the original system when initial data or viscosity coefficients are unknown, by incorporating a feedback control term generated by observational data over a time period. In this study, based on a regularized form derived from the variational inequalities of the Navier-Stokes equations  with nonlinear slip boundary conditions, we first investigate the classical CDA problem when initial data is absent. After establishing the existence, uniqueness and regularity of the solution, we prove its exponential convergence with respect to the time. Additionally, we extend the CDA to address the problem of missing viscosity coefficients and analyze its convergence order, too. Furthermore, utilizing the predictive capabilities of partial evolutionary tensor neural networks (pETNNs) for time-dependent problems, we propose a novel CDA by replacing observational data with predictions got by pETNNs. Compared with the classical CDA, the new one can achieve similar approximation accuracy but need much less computational cost. Some numerical experiments are presented, which not only validate the theoretical results, but also demonstrate the efficiency of the CDA.
\end{abstract}




\end{frontmatter}
\section{Introduction}
Time-dependent partial differential equations (PDEs) play a crucial role in modeling the dynamic evolution of physical systems. Accurate numerical solutions to these equations rely heavily on the precise definition conditions, such as initial data, physical coefficients and boundary conditions. However, in many real-world applications, such information is often incomplete, sparse, or of low resolution, posing significant challenges to traditional simulation methods. Continuous data assimilation (CDA) offers a powerful approach to address these challenges by incorporating observational data into the models, thereby recovering the original system and enhancing solution accuracy, which have been widely applied in various fields, such as meteorology and fluid simulation (see, e.g.,  \cite{aref7,ref3,ref2,aref6}).

In \cite{ref3}, through adding a feedback control term into the equations, Azouani, Olson, and Titi studied the CDA for  the two-dimensional Navier-Stokes  equations with incomplete initial data,  which was widely known as the AOT method.  The uniqueness and asymptotic convergence of the CDA system were established. Subsequently,  Biswas and Price \cite{ref4} extended the AOT method to three-dimensional Navier-Stokes equations and proved the global well-posedness and asymptotic stability of the method.  Further investigations were made by Gardner et al. \cite{ref5}, who explored  the CDA for the velocity-vorticity formulation of the Navier-Stokes equations, revealing that the integration of velocity and vorticity data enhances long-term stability. More recently, the CDA has also been applied to specialized models, such as You et al. \cite{aref6} demonstrated the CDA's capability to assimilate data in large-scale ocean circulation models even with systematic errors and lead to exponential convergence to the true solution, and Akbas et al. \cite{aref7} used the CDA for barotropic vorticity models and showed that the optimal accuracy could be achieved in sparse observational conditions. Simultaneously, the influence of large nudging parameters on the accuracy and convergence rate of the continuous data assimilation (CDA) had been investigated. Diegel et al. \cite{aref8} demonstrated that large nudging parameters could optimize the long-term accuracy of the Navier-Stokes equations without  introducing scaling errors. Carlson et al. \cite{aref9} claimed that,as the nudging parameter increases, the nudging filter tends to synchronize, providing valuable insights for the development of adaptive strategies.  Larios et al. \cite{aref10} observed exponential decay of errors in the ocean model even when faced with challenges such as vertical mixing. These findings offer critical insights for adapting the CDA to more complex, high-resolution systems. In addition,  \cite{aref11} introduced a nonlinear-nudging version of the CDA for the two-dimensional Navier-Stokes equations and achieved super-exponential convergence once the error falls below a threshold.   Furthermore, Larios et al. \cite{aref12} improved the convergence rate for sparse-in-time observations by introducing a ``data assimilation window", which means  the system can be relaxed after each observation. Meanwhile, Carlson et al. \cite{ref15} successfully recovered the viscosity coefficient for the two-dimensional Navier-Stokes equations via the CDA and rigorously established the order of convergence. Nevertheless, in all referred works above, the data assimilation is based on the observational data over a time period, which is limited by the physical  experiment environment and has high computational cost in the long time simulations.


On the other hand, designing efficient numerical methods for solving the Navier-Stokes equations with nonlinear slip boundary conditions, which are widely used to simulate fluid-solid wall interactions (see, e.g., \cite{ref6}), has also attracted many attentions.  Slip boundary conditions are typically expressed as friction-type constraints and are often formulated as a class of variational inequalities. These boundary conditions introduce nonlinearity, which increases the complexity in searching numerical methods and deducing the theoretical analysis. To address slip boundary conditions, researchers have proposed the penalty method, which transforms the inequality constraint on the boundary into an equality constraint by introducing a penalty parameter \cite{ref7}. This method allows the boundary conditions to be decoupled under certain conditions, making it suitable for finite element methods. However, the convergence speed and accuracy of the penalty method suffer in complex boundary conditions, particularly in high-dimensional cases where computational costs become significantly higher. Another commonly used method is the pressure projection method \cite{ref8}, which decouples the pressure and velocity terms in the Navier-Stokes equations to reduce the complexity of nonlinear boundary conditions. Li et al. \cite{ref9} proposed an improved approach based on the pressure projection. Moreover, some other methods were also investigated for this problem, such as the authors studied two-grid variational multiscale algorithms for the high Reynolds number problem in \cite{AMZ2017}, \cite{D2017}  focused on the convergence of the nonconforming finite element method,  \cite{ZS2022} was devoted to a three-step defect-correction algorithm and this idea was combined with the parallel algorithm with a fully overlapping domain decomposition in \cite{ZS2023}. However, all these results were built on the equations with precise definition conditions.

 At the same time, it must be noted that deep neural networks (DNNs) were widely used to solve PDEs in the past several years. In 2019, Raissi et al. \cite{ref10} introduced Physics-Informed Neural Networks (PINNs), which embed physical laws in the loss function to approximate PDEs' solutions with limited data, improving efficiency and precision. PINNs have since been extended to address various challenges. Sarma et al. \cite{refn1} developed Interface PINNs (I-PINNs), which use separate networks for each subdomain in interface problems, enhancing accuracy and reducing computational costs. Zou et al. \cite{refn2} tackled the model's misspecification by introducing correction mechanisms with Bayesian PINNs (B-PINNs), quantifying uncertainties in the governing equations. Liu et al. \cite{refn3} proposed PINNs-WE (PINNs with equation weight), improving the shock wave capture by adding physics-dependent weights and using Rankine-Hugoniot relations near discontinuities. Zhang et al. \cite{refn4} combined PINNs with the  symbolic regression to discover a reaction-diffusion model for tau protein misfolding in Alzheimer's, demonstrating PINNs' potential in biological modeling. Song et al. \cite{refn5} introduced VW-PINNs (volume-weighting PINNs), which improve the convergence by volume-weighting PDE residuals in nonuniform sampling methods, enhancing both accuracy and efficiency in fluid mechanics. In addition, the deep Galerkin method (DGM) studied in \cite{ref16} extended the use of neural networks by solving Stokes equations without the mesh generation, ensuring boundary conditions through random sampling. This method offers computational efficiency and guarantees the convergence in high-dimensional problems. In 2023, Peng et al. \cite{ref17} proposed a non-gradient method based on the Neural Tangent Kernel (NTK) for solving elliptic PDEs. This method replaces traditional gradient descent with dissipative dynamics, accelerating the convergence and reducing computational costs, especially for solutions with high-frequency components. In 2024, Wang et al. \cite{refn7} introduced a tensor neural network (TNN)-based method for solving high-dimensional PDEs. By incorporating a posteriori error estimators, this approach improves the accuracy of high-dimensional problems, particularly those involving second-order elliptic operators, eigenvalue problems, and boundary value problems. Their method addresses the ``curse of dimensionality" and achieves stable convergence with high efficiency. Furthermore, Kao et al. \cite{ref14} introduced partial evolutionary tensor neural networks (pETNNs), combining tensor neural networks with evolutionary deep neural networks to solve time-dependent PDEs. By updating only partial parameters, pETNNs significantly reduce computational costs, showing strong performance in high-dimensional problems. However, the predicted values from neural networks typically do not achieve the high-precision approximation effects as that obtained through traditional numerical methods, especially for the problem with incomplete definite conditions.

In this paper, we extend the application of the CDA  to the Navier-Stokes equations with nonlinear slip boundary conditions which have incomplete definition conditions. We will first establish the existence, uniqueness and convergence of the CDA problem  in the absence of initial data. Then, we will study the CDA problem with missing viscosity coefficients and deduce its error with respect to the viscosity. Moreover, recognizing the limitation on both observational data from the physical experiment environment in the CDA and the precision generated by pETNNs, we proposed a novel CDA which combines the strengths of CDA and pETNNs and provide high-precision approximations regardless of observational data from the physical experiment environment for the Navier-Stokes equations with nonlinear slip boundary conditions.  The correctness and effectiveness of the considered method is validated through some numerical experiments.

This paper is organized as follows:  In Section 2, some function spaces and the regularization form of Navier-Stokes equations with nonlinear slip boundary conditions are introduced, which lay the foundation for the CDA. Then, we focus on the classical CDA in Section 3. After deducing the convergence for the cases of missing initial data and unknown viscosity coefficients, we consider the finite element approximation of the CDA, including several numerical experiments to verify the theoretical predictions. A new CDA framework based on pETNNs is proposed in Section 4, which integrates deep learning prediction and feedback control mechanisms to achieve efficient observational data and error correction.  Finally, conclusions are made in Section 5.

\section{Preliminaries}
In this section, after introducing some preliminaries, we will present the regularization approximation of the Navier-Stokes equations with nonlinear slip boundary conditions, which is very usefully in the CDA.
\subsection{Equations and functional spaces}
In this paper, we consider the Navier-Stokes equations with nonlinear slip boundary conditions as follows
\begin{align}
    & u_t  - 2\nu\nabla\cdot D(u) + (u\cdot \nabla)u + \nabla p = f,\quad& \mathrm{in} \ \Omega\times(0,T],\label{1}\\
    &\nabla\cdot u = 0, \quad& \mathrm{in}\ \Omega\times(0,T],\label{2}\\
    &u(0)=u_0.\quad&\mathrm{in}\ \Omega,\label{3}\\
    &u = 0,\quad& \mathrm{on} \ \Gamma\times[0,T],\label{4}\\
    &u_n = 0,\quad -\sigma_{\tau}(u)\in g\partial |u_{\tau}|,\quad& \mathrm{on} \ S\times[0,T].\label{5}
\end{align}
Here $\Omega\in \mathbb{R}^2 $ is   open, bounded and connected set,  $\partial\Omega$ is the boundary which is continuous and divided into $\Gamma$ and $S$ satisfying $\Gamma\cap S = \phi,\ |\Gamma|\neq 0,\ |S|\neq 0\ $and $\overline{\Gamma\cup S} = \partial\Omega$. $T\geq 0$ is the final time. The function $u=u(\textbf{x}, t)=(u_1(\textbf{x}, t),u_2(\textbf{x}, t))^\top$  represents the velocity with $u_n = u\cdot n$  and $u_{\tau} = u\cdot \tau$ denoting its normal and tangential components where $n$ and $\tau$ are the unit normal and tangential vectors, respectively. The function $p=p(\textbf{x}, t)$ is the pressure and $f$ is the body force. In addition, $D(u)=\frac{\nabla u+\nabla u^\top}{2}$ is the rate of deformation tensor, $\sigma_{\tau}$ represents the tangential component of the stress tensor  $\sigma(u,p) = (2\nu D(u)-pI)\cdot n$ where $I$ is the identity matrix of the second order,  $g\in L^2(S)$ is a non-negative function and the subdifferential operator $\partial |\psi|$ is defined as follows
\begin{equation}
    \partial|\psi| = \left\{ r\in \mathbb{R}: r(l-\psi)\leq |l|-|\psi| ,\quad\forall l\in \mathbb{R} \right\}.\label{6}
\end{equation}

Before proceeding the analysis in the following, we firstly recall some notations in spaces and norms. Set
\[
    \begin{aligned}
        &V = \left\{ u\in H^1(\Omega)^2,\ u|_{\Gamma}=0,\ u\cdot n|_S =0 \right\},\\
        &V_{\sigma} = \left\{ u\in V,\ \nabla\cdot u = 0 \right\},\quad V_0 = H_0^1(\Omega)^2,\\
        &H = \left\{ u\in L^2(\Omega)^2,\ \nabla\cdot u = 0,\ u\cdot n|_{\partial\Omega}=0 \right\},\\
        &M = L_0^2(\Omega) = \left\{ q\in L^2(\Omega),\ \int_{\Omega}q\mathrm{d}\textbf{x}=0 \right\}.
    \end{aligned}
\]
$(\cdot,\cdot)$ and $||\cdot||$ denote  the inner product and the norm in $L^2(\Omega)$ or $L^2(\Omega)^2$, while define the inner product and the norm in $V$ by $(\nabla\cdot, \nabla\cdot)$ and $||\cdot||_V = ||\nabla\cdot||$, respectively. Obviously, $||\cdot||_V$ is equivalent to $||\cdot||_1$. In addition, we equip the norm in the Hilbert space $H^k(\Omega)$ or $H^k(\Omega)^2$ by $||\cdot||_k$.

Then, define the Leray projector $P_{\kappa }$ as the orthogonal projection from $L^2(\Omega)^2$ onto
$H$,  the Stokes operator $A: V \to V^*$ ($V^*$ is the dual space of V), and the bilinear term $B: V \times V \to V^*$ to be the continuous extensions of the operators given by
\[
    Au = -2P_{\kappa }\nabla\cdot D(u),\quad B(u,v) = P_{\kappa }(u\cdot\nabla v)
\]
with the domain of $A$ by $\mathcal{D}(A) =\{u\in V:Au\in H\}$.
\par
Moreover, we define the following bilinear and trilinear forms
\[\begin{aligned}
        &a(u,v) = (\nabla u,\nabla v) = \int_{\Omega}\nabla u\cdot \nabla v\mathrm{d}\textbf{x},\quad \forall u, v\in V,\\
        &b(u,v,w) = (B(u,v),w) = \int_{\Omega}(u\cdot\nabla)v\cdot w\mathrm{d}\textbf{x},\quad\forall u, v, w\in V,\\
        &d(v,p) = (p,\nabla\cdot v) = \int_{\Omega} p\nabla \cdot v\mathrm{d}\textbf{x},\quad\forall p\in M, v\in V.
\end{aligned}\]
For the bilinear form $d(\cdot,\cdot)$, there holds the inf-sup condition \cite{insu}
\begin{align}\label{10}
\sup\limits_{v\in V}\frac{d(v,p)}{||v||_V}\geq c||p||,\qquad \forall p\in M,
\end{align}
where $c$ is a positive constant. For the trilinear form, we have \cite{ref8}
\begin{align}
    & b(u,v,w) = -b(u,w,v),\quad \forall u\in V_{\sigma},\ v,w\in V,\label{11}\\
    &b(u,v,v) = 0,\quad \forall u\in V_{\sigma},\ v\in V,\label{12}\\
    &|b(u,v,w)|\leq c_1||u||^{\frac{1}{2}}||u||_V^{\frac{1}{2}}||v||_V^{\frac{1}{2}}||Av||^{\frac{1}{2}}||w||,\ \ u\in V,\ v\in \mathcal{D}(A),\ w\in H, \label{13}\\
    &|b(u,v,w)|\leq c_2||u||^{\frac{1}{2}}||u||_V^{\frac{1}{2}}||v||_V||w||^{\frac{1}{2}}||w||_V^{\frac{1}{2}},\ \  u,v,w\in V.\label{14}
\end{align}
\subsection{Regularization approximation}
Given $f\in L^2(\Omega)^2$ and $g\in L^2(S)$ with $g\geq 0 \ \mathrm{on} \ S$, we will firstly derive the weak variational formulation of (\ref{1})-(\ref{5}). Multiplying (\ref{1}) by $v-u$ for $\forall v\in V$ and integrating over $\Omega$, we obtain
\begin{equation}\label{15}
    ( u_t ,v-u) + \nu a(u,v-u) + b(u,u,v-u) - d(v-u) - \int_S \sigma\cdot (v-u)\mathrm{d}s = (f,v-u).
\end{equation}
Decompose $\sigma$ and $v-u$ into its normal and tangential components, i.e.,
\begin{align*}
    &\sigma = \sigma_n\cdot n+\sigma_{\tau}\cdot \tau,\\
    &v-u = (v_n-u_n)\cdot n+(v_{\tau}-u_{\tau})\cdot\tau,
\end{align*}
it follows that
\[\begin{aligned}
    \int_S \sigma\cdot(v-u)\mathrm{d}s &= \int_S (\sigma_n\cdot n+\sigma_{\tau}\cdot \tau)\left((v_n-u_n)\cdot n+(v_{\tau}-u_{\tau})\cdot\tau\right)\mathrm{d}s\\
    &= \int_S \sigma_{n}(v_n-u_n)\mathrm{d}s + \int_S \sigma_{\tau}(v_{\tau}-u_{\tau})\mathrm{d}s\\
    &= \int_S \sigma_{\tau}(v_{\tau}-u_{\tau})\mathrm{d}s.
\end{aligned}\]
On the other hand, from (\ref{5}) and (\ref{6}), there exists $r\in \mathbb{R}$,  we have, by setting $l = v_{\tau}$, that
\[-\sigma_{\tau}(v_{\tau}-u_{\tau}) = gr(v_{\tau}-u_{\tau})\leq g(|v_{\tau}|-|u_{\tau}|),\]
which implies
\begin{equation}\label{16}
    -\int_S \sigma\cdot(v-u)\mathrm{d}s = -\int_S\sigma_\tau(v_\tau-u_\tau)\mathrm{d}s \leq \int_S g(|v_{\tau}|-|u_{\tau}|)\mathrm{d}s.
\end{equation}
Putting (\ref{16}) into (\ref{15}), we obtain the variational formulation of (\ref{1})-(\ref{5}) as follows: find $(u,p)\in V\times M$ such that
\begin{equation}\label{17}
\left\{\begin{aligned}
    &(u_t ,v-u) + \nu a(u,v-u) + b(u,u,v-u) - d(v-u,p) + j(v_{\tau}) - j(u_{\tau}) \geq (f,v-u),\\
    &d(u,q) = 0,
    \end{aligned}\right.
\end{equation}
for $\forall \ (v,q)\in V\times M,$ where $j(\eta) = \int_S g|\eta|\mathrm{d}s$ and the second equation in the above is got by multiplying both sides of (\ref{2}) by $q\in M$.

For (\ref{17}), since the function $j(\eta)$  is not differentiable, it is not convenient for analyzing. Next, we will convert this inequality into an equation by using the regularization method. To approximate the function $j(\eta)$, the following convex differentiable function $j_{\varepsilon}(\eta)$ is used
\begin{equation}\label{jVare} j_{\varepsilon}(\eta) = \int_S g\sqrt{\eta^2+\varepsilon^2}\mathrm{d}s,\qquad \forall \eta \in \mathbb{R},\end{equation}
where $\varepsilon$ is a penalty parameter which is sufficient small. It is easy to check that
\begin{equation}\label{18}
    \left|
        j_{\varepsilon}(\eta)-j(\eta)
    \right| \leq \varepsilon\int_S g\mathrm{d}s,\qquad \forall \eta\in  \mathbb{R},
\end{equation}
and
\begin{equation}\label{19}
\int_S\nabla j_{\varepsilon}(\eta)\xi ds=\lim\limits_{h\rightarrow 0}\frac{j_{\varepsilon}(\eta+h\xi)-j_{\varepsilon}(\eta)}{h}=\int_Sg\frac{\eta\xi}{\sqrt{\eta^2+\varepsilon^2}}ds, \qquad \forall \xi,\eta\in  \mathbb{R}.
\end{equation}
Approximating $j(\eta)$  with $j_{\varepsilon}(\eta)$  in the variational form (\ref{17}) yields: find $(u^{\varepsilon},p^{\varepsilon})\in V\times M$ such that
\begin{equation}\label{20}
   \left\{\begin{aligned}
    &(u^{\varepsilon}_t,v-u^{\varepsilon}) + \nu a(u^{\varepsilon},v-u^{\varepsilon}) + b(u^{\varepsilon},u^{\varepsilon},v-u^{\varepsilon}) - d(v-u^{\varepsilon},p^{\varepsilon}) + j_{\varepsilon}(v_{\tau}) \\
    &\qquad\qquad\qquad\qquad\qquad\qquad\qquad\qquad\qquad\qquad- j_{\varepsilon}(u_{\tau}^{\varepsilon}) \geq (f,v-u^{\varepsilon}),\\
    &d(u^{\varepsilon},q) = 0,
    \end{aligned}\right.
\end{equation}
for $\forall (v,q)\in V\times M$. Noting that $j_\varepsilon(\eta)$ is differentiable, using (\ref{19}) and the standard convex analysis, we can prove that (\ref{20}) is equivalent to the following variational problem \cite{ref8}: find $(u^{\varepsilon},p^{\varepsilon})\in V\times M$ such that
\begin{equation}\label{21}
\left\{\begin{aligned}
        &(u_t^{\varepsilon},v-u^{\varepsilon}) + a(u^{\varepsilon},v-u^{\varepsilon}) + b(u^{\varepsilon},u^{\varepsilon},v-u^{\varepsilon}) - d(v-u^{\varepsilon},p^{\varepsilon}) +  (\beta(u_{\tau}^{\varepsilon}),v_{\tau})_S= (f,v-u^{\varepsilon}),\\
        &d(u^{\varepsilon},q) = 0,
\end{aligned}\right.\end{equation}
for $\forall (v,q)\in V\times M$ where $\beta(u_\tau^\varepsilon)=g\frac{u_\tau^\varepsilon}{\sqrt{(u_\tau^\varepsilon)^2+\varepsilon^2}}$ and
$$(\beta(u_{\tau}^{\varepsilon}),v_{\tau})_S = \int_S g\frac{u_{\tau}^{\varepsilon}v_{\tau}}{\sqrt{(u_{\tau}^{\varepsilon})^2+\varepsilon^2}}ds.$$
If the boundary $\partial \Omega$ is sufficiently smooth and the solution $(u^{\varepsilon}, p^{\varepsilon})$ of the regularized problem (\ref{21}) is sufficiently smooth \big(e.g., $u^{\varepsilon} \in C^2(\Omega) \cap C^1(\bar{\Omega})$ and $p^{\varepsilon} \in C^1(\Omega)\cap C(\bar{\Omega})$\big),  then $(u^{\varepsilon}, p^{\varepsilon})$ satisfies:
\begin{equation}\label{22}
\left\{\begin{aligned}
    &u_t^{\varepsilon}-2\nu\nabla\cdot D(u^{\varepsilon})+(u^{\varepsilon}\cdot \nabla)u^{\varepsilon}+\nabla p^{\varepsilon} = f,\quad&\mathrm{in} \ \Omega\times(0,T],\\
    &\nabla\cdot u^{\varepsilon} = 0,\quad&\mathrm{in}\ \Omega\times(0,T],\\
    &u^{\varepsilon}(0)=u_0,\quad &\mathrm{in} \ \Omega,\\
    &u^{\varepsilon}=0, &\mathrm{on} \ \Gamma\times [0,T],\\
    &u^{\varepsilon}\cdot n =0,\ -\sigma_{\tau}(u^{\varepsilon})=\beta(u^{\varepsilon}_{\tau}),\quad &\mathrm{on} \ S\times[0,T].
\end{aligned}\right.
\end{equation}

Moreover, thanks to (\ref{10}), (\ref{22}) is equivalent to \cite{ref9}
\begin{equation}\label{23}
\left\{\begin{aligned}
    &u_t^\varepsilon+\nu A u^{\varepsilon}+B(u^{\varepsilon},u^{\varepsilon}) = f,\quad&\mathrm{in} \ \Omega\times(0,T],\\
    &u^{\varepsilon}(0)=u_0,\quad &\mathrm{in} \ \ \Omega,\\
    &u^{\varepsilon}=0,&\mathrm{on} \ \Gamma\times[0,T],\\
    &u^{\varepsilon}\cdot n =0,\ -2D_{\tau}(u^{\varepsilon})=\beta(u^{\varepsilon}_{\tau})\quad &\mathrm{on} \ S\times[0,T].
\end{aligned}\right.
\end{equation}

\begin{theorem}\label{thmnew}
    Assume that $u_0 \in H, f \in L^2\big([0,T];V^*\big)$ and $g \in L^2\big([0,T];L^2(S)\big)$ with $g \geq 0$. Then the problem (\ref{23}) admits a unique solution such that
    \begin{equation}\label{24}
        ||u^{\varepsilon}(t)||^2+\nu\int_{0}^{T}||u^{\varepsilon}(\varsigma)||_V^2\mathrm{d}\varsigma+\int_{0}^{T}(\beta(u_\tau^\varepsilon),u_\tau^\varepsilon)_S\mathrm{d}\varsigma\leq ||u_0||^2+\frac{1}{2\nu}||f||^2_{L^2([0,T];V^*)}.
    \end{equation}
    Furthermore, if $u_0\in \mathcal{D}(A),\ f\in L^\infty([0,T];H),\ \int_{S}g\sqrt{u_0^2+\varepsilon^2}\leq c$ and $g\in H^\frac{1}{2}(S)$ with $g\geq 0$, then, for a fixed $T^*>0$, there holds
    \begin{align}
        \int_{t}^{t+T^*}||u^\varepsilon(\varsigma)||_V^2\;\mathrm{d}\varsigma\leq \frac{||u_0||^2}{2\nu}+(1+2T^*)G,\label{24b}\\
        ||u^\varepsilon(t)||^2+||u^\varepsilon(t)||_V^2+||Au^\varepsilon(t)||^2\leq c,\label{24c}
    \end{align}
    where $t+T^*\leq T$ and $G$ is defined as
    \begin{equation}\label{G}
        G=\frac{1}{4\nu^2}\limsup_{t\to\infty}||f||.
    \end{equation}
\end{theorem}
\begin{proof}\label{a1}
    The uniqueness and (\ref{24}) and (\ref{24c}) have been proved in \cite{ref9}.

    Taking the inner product with $u^\varepsilon$ for the first equation in (\ref{23}) and using (\ref{12}) yields
    \begin{equation}\label{a2}
        \frac{1}{2}\frac{\mathrm{d}}{\mathrm{d}t}||u^\varepsilon||^2+\nu(Au^\varepsilon,u^\varepsilon)=(f,u^\varepsilon)
    \end{equation}
    Since
    \begin{equation}\label{a3}
        \begin{aligned}
            \nu(Au^{\varepsilon},u^{\varepsilon})&=-2\nu\int_{\Omega}\nabla\cdot D(u^{\varepsilon})\cdot u^{\varepsilon} \mathrm{d}\textbf{x}\\
            &=\nu\int_{\Omega}\nabla u^{\varepsilon}\cdot\nabla u^{\varepsilon}\mathrm{d}\textbf{x}- 2\nu\int_{\partial\Omega}u^{\varepsilon}\cdot(D(u^{\varepsilon})\cdot n)\mathrm{d}s\\
            &=\nu\int_{\Omega}\nabla u^{\varepsilon}\cdot\nabla u^{\varepsilon}\mathrm{d}\textbf{x}+\int_S \beta(u^{\varepsilon}_{\tau})\cdot u^{\varepsilon}_{\tau}\mathrm{d}s\\
            &=\nu ||u^{\varepsilon}||_V^2+\big(\beta(u^{\varepsilon}_{\tau}),u^{\varepsilon}_{\tau}\big)_S,
        \end{aligned}
    \end{equation}
    and
    \[\begin{aligned}
        (f,u^{\varepsilon})\leq||f||\ ||u^\varepsilon||\leq\frac{1}{2\nu}||f||^2+\frac{\nu}{2}||u^\varepsilon||^2.
    \end{aligned}\]
    Putting these inequalities into (\ref{a2}) and integrating form $t$ to $t+T^*$ with respect to time, we get
    \begin{align*}
        ||u^\varepsilon(t+T^*)||^2+2\nu\int_{t}^{t+T^*}||u^\varepsilon||_V^2\mathrm{d}\varsigma+\int_{t}^{t+T^*}(\beta(u_\tau^\varepsilon),u_\tau^\varepsilon)_S\mathrm{d}\varsigma\\
        \leq ||u^\varepsilon(t)||^2+\nu\int_{t}^{t+T^*}||u^\varepsilon||^2\mathrm{d}\varsigma+\frac{1}{\nu}\int_{t}^{t+T^*}||f||^2\mathrm{d}\varepsilon
    \end{align*}
    Combining (\ref{24}) with the non-negativity of $\big(\beta(u^{\varepsilon}_{\tau}),u^{\varepsilon}_{\tau}\big)_S$ and $f\in L^\infty([0,T];H)$, we obtain
    \[
    \int_{t}^{t+T^*}||u^\varepsilon(\varsigma)||_V^2\;\mathrm{d}\varsigma\leq \frac{||u_0||^2}{2\nu}+(1+2T^*)G.
    \]

\end{proof}
To facilitate the subsequent proofs, we recall the Gronwall inequality at the end of this section.
\begin{lemma}(\cite{ref19} Generalized uniform Gronwall inequality)
    Let $\alpha$ be a locally integrable real-valued function defined on $(0,\infty)$, satisfying the following conditions for some $0<\mathfrak{T} <\infty$:
    \[
        \begin{aligned}
            &\liminf_{t\to \infty}\int_{t}^{t+\mathfrak{T}}\alpha(\varsigma )\mathrm{d}\varsigma  =\gamma>0,\\
            &\limsup_{t\to \infty}\int_{t}^{t+\mathfrak{T}}\alpha^-(\varsigma )\mathrm{d}\varsigma  ={\Lambda}<\infty.
        \end{aligned}
    \]
    where $\alpha^- = \max\{-\alpha, 0\}.$ Furthermore, let $\phi$ be a locally integrable real-valued function defined on $(0,\infty)$, and let $\phi ^+ = \max\{\phi, 0\}$. Suppose that $Y$ is an absolutely continuous nonnegative function on $(0,\infty)$ such that
    \[\frac{\mathrm{d}Y}{\mathrm{d}t} +\alpha Y\leq \phi \ \ \ a.e. \ on\ (0,\infty).\]
    Then
    \begin{equation}\label{lem1}
        Y(t)\leq Y(t_0)\hat{\Lambda}e^{-\frac{\gamma}{2\mathfrak{T}}(t-t_0)}+\left(\sup_{t\geq t_0}\int_{t}^{t+\mathfrak{T}}\phi ^+(\varsigma )\mathrm{d}\varsigma \right)\hat{\Lambda}\frac{e^{\gamma/2}}{e-1},
    \end{equation}
    where $\hat{\Lambda} = e^{\Lambda+1+\gamma/2}$ and $t_0$ is chosen sufficiently large so that for all $s>t_0$,
    \[
        \int_{t}^{t+\mathfrak{T}}\alpha^-(\varsigma )\mathrm{d }\varsigma \leq{\Lambda}+1
    \]
    and
    \[
        \int_{t}^{t+\mathfrak{T}}\alpha(\varsigma )\mathrm{d}\varsigma \geq\gamma/2.
    \]
    Particularly, when $\phi=0$, it yields
    \begin{equation}\label{lem2}
        Y(t)\leq Y(t_0)\hat{\Lambda}e^{-\frac{\gamma}{2\mathfrak{T}}(t-t_0)}.
    \end{equation}
\end{lemma}
\section{Classical Continuous Data Assimilation}
In this section, we will focus on the continuous data assimilation for the Navier-Stokes equations with nonlinear slip boundary conditions. Within the framework of classical numerical methods, simulations of the problems are often impossible when the initial data or the viscosity coefficient is unavailable. However, by constructing a feedback control term based on velocity observations over a certain time period to modify the model, simulations become feasible even in the absence of some essential conditions. This is the core idea of the continuous data assimilation method. For the Navier-Stokes equations with Dirichlet and periodic boundary conditions, it has been proven that the solutions of the continuous data assimilation model converge exponentially to the exact ones. In this work, we extend this idea to the Navier-Stokes equations with nonlinear slip boundary conditions. We will investigate the continuous data assimilation problems with missing initial data in Section 3.1 and missing viscosity coefficient in Section 3.2, respectively. Subsequently, in Section 3.3, we will explore the finite element approximation of the equations. Finally, in Section 3.4, we will present some numerical examples to verify the theoretical predictions.
\subsection{Problem with initial data missing}
In this subsection, we focus on the continuous data assimilation of problem (\ref{22}) with initial data $u_0$ missing. Let $I_{\hat{h}}$ be a general interpolation operator which is only relevant to the spatial accuracy $\hat{h}$ of the observation locations and satisfying {(see, e.g., \cite{ref3})}
\begin{equation}\label{25}
    ||\varphi-I_{\hat{h}}(\varphi)||^2 \leq c_0 \hat{h}^2||\varphi||_1^2,\quad \forall \varphi\in H^1(\Omega)^2.
\end{equation}
Assume that the observational data with respect to $u^{\varepsilon}$ at the locations labeled by spatial size $\hat{h}$ are available on $(0, T]$, then the continuous data assimilation problem of the Navier-Stokes equations (\ref{22}) with regularized nonlinear slip boundary conditions is as follows: find $v^{\varepsilon}\times p^{\varepsilon}\in V\times M$ such that
\begin{equation}\label{26}\left\{\begin{aligned}
    &v_t^{\varepsilon}-\nu\nabla\cdot D(v^{\varepsilon})+(v^{\varepsilon}\cdot \nabla)v^{\varepsilon}+\nabla p^{\varepsilon} = f+\mu(I_{\hat{h}}(u^{\varepsilon})-I_{\hat{h}}(v^{\varepsilon}))\quad&\mathrm{in} \ \Omega\times(0,T],\\
    &\nabla\cdot v^{\varepsilon} = 0,\quad&\mathrm{in}\ \Omega\times(0,T],\\
    &v^{\varepsilon}(0)=v_0,\quad &\mathrm{in} \ \ \Omega,\\
    &v^{\varepsilon}=0, &\mathrm{on} \ \Gamma\times[0,T],\\
    &v^{\varepsilon}\cdot n =0,\ -\sigma_{\tau}(v^{\varepsilon})=\beta(v^{\varepsilon}_{\tau})\quad &\mathrm{on} \ S\times[0,T].
\end{aligned}\right.\end{equation}
where $\mu(I_{\hat{h}}(u^{\varepsilon})-I_{\hat{h}}(v^{\varepsilon}))$ is the introduced feedback control term with $\mu$ being a relaxation factor and $v_0$ is a ``guessed" initial data of the system.
\par
Similar to (\ref{23}), the above continuous data assimilation problem can be rewritten as
\begin{equation}\label{27}
\left\{\begin{aligned}
    &v_t^\varepsilon+\nu A v^{\varepsilon}+B(v^{\varepsilon},v^{\varepsilon}) = f+\mu P_{\kappa }(I_{\hat{h}}(u^{\varepsilon})-I_{\hat{h}}(v^{\varepsilon}))\quad&\mathrm{in} \ \Omega\times(0,T],\\
    &v^{\varepsilon}(0)=v_0,\quad &\mathrm{in} \ \ \Omega,\\
    &v^{\varepsilon}=0,&\mathrm{on} \ \Gamma\times[0,T],\\
    &v^{\varepsilon}\cdot n =0,\ -2D_\tau(v^{\varepsilon})=\beta(v^{\varepsilon}_{\tau})\quad &\mathrm{on} \ S\times[0,T].
\end{aligned}\right.
\end{equation}

\begin{theorem}
    Suppose that $f\in L^2([0,T]; V^*)$, $\ g\in L^2([0,T]; L^2(S))$ is non-negative, $I_{\hat{h}}$ satisfies (\ref{25}) and $\mu c_0 {\hat{h}}^2\leq \nu$ where $c_0$ is the constant appearing in (\ref{25}). Then the continuous data assimilation equations (\ref{27}) has a unique strong solution $v^\varepsilon$ that satisfies
    \begin{equation}\label{28}
        v^{\varepsilon}\in C([0,T],V)\cap L^2\big((0,T),\mathcal{D}(A)\big) \ and \  v_t^\varepsilon\in L^2\big((0,T),H\big).
    \end{equation}
\end{theorem}
\begin{proof}
    Define $\kappa:=f+\mu P_{\kappa }I_{\hat{h}}(u^{\varepsilon})$. Since $u^{\varepsilon} \in L^{\infty}\big([0,T], H)\cap L^2\big([0,T],V\big)$ (see Theorem 1),  there exists a constant $M>0$, such that
    \begin{equation}\label{kappa}\begin{aligned}
        ||\kappa||&\leq ||f|| +\mu ||P_{\kappa }\big(I_{\hat{h}}(u^{\varepsilon})\big)||\\
               &\leq ||f|| +\mu ||P_{\kappa }\big(u^{\varepsilon}-I_{\hat{h}}(u^{\varepsilon})\big)||+\mu ||u^{\varepsilon}||\\
               &\leq ||f|| +\mu{\sqrt{c_0}\hat{h}} ||u^{\varepsilon}||_1 +\mu||u^{\varepsilon}||\\
               &<M,
    \end{aligned}\end{equation}
   by utilizing (\ref{25}). Next, we show the existence of solutions $v^{\varepsilon}$ to (\ref{27}) by using the Galerkin method. Let $P_m$ be the $m$-th Galerkin projector and $v^{\varepsilon,m}$ be the solution to the finite-dimensional Galerkin truncation
    \begin{equation}\label{29}
    \left\{\begin{aligned}
        &\frac{\mathrm{d}v^{\varepsilon,m}}{\mathrm{d}t}+\nu A v^{\varepsilon,m}+P_m B(v^{\varepsilon,m},v^{\varepsilon,m}) = P_m \kappa-\mu P_m I_{\hat{h}}(v^{\varepsilon,m})\quad&\mathrm{in} \ \Omega\times(0,T],\\
        &v^{\varepsilon,m}(0)=P_m v_0\quad &\mathrm{in} \ \ \Omega,\\
        &v^{\varepsilon}=0,&\mathrm{on} \ \Gamma\times[0,T],\\
        &v^{\varepsilon,m}\cdot n =0,\ -2D_{\tau}(v^{\varepsilon,m})=\beta(v^{\varepsilon,m}_{\tau})\quad &\mathrm{on} \ S\times[0,T].
    \end{aligned}\right.\end{equation}
    We can find that (\ref{29}) is a finite system of ordinary differential equations, which has short time existence and uniqueness.  Focusing on the maximal interval of existence $[0,T_m)$, we will show the uniform bounds for $v^{\varepsilon,m}$, which are independent of $m$. This in turn will imply the global existence for (\ref{29}).

    Taking the inner product of the first equation in (\ref{29}) with $v^{\varepsilon,m}$ and using (\ref{12}), we obtain
    \begin{equation}\label{30}
        \frac{1}{2}\frac{\mathrm{d}}{\mathrm{d}t}||v^{\varepsilon,m}||^2+\nu (Av^{\varepsilon,m},v^{\varepsilon,m})=(\kappa,v^{\varepsilon,m})-\mu (I_{\hat{h}}(v^{\varepsilon,m}),v^{\varepsilon,m}).
    \end{equation}
    Similar to (\ref{a3}) in theorem \ref{thmnew}, we get
    \[
        \begin{aligned}
            \nu(Av^{\varepsilon,m},v^{\varepsilon,m})&=\nu ||v^{\varepsilon,m}||_V^2+\big(\beta(v^{\varepsilon,m}_{\tau}),v^{\varepsilon,m}_{\tau}\big)_S,
        \end{aligned}
    \]
    putting this relation into (\ref{30}) and using (\ref{25}), we obtain
    \[\begin{aligned}
        &\qquad\frac{1}{2}\frac{\mathrm{d}}{\mathrm{d}t}||v^{\varepsilon,m}||^2+\nu ||v^{\varepsilon,m}||_V^2+\big(\beta(v^{\varepsilon,m}_{\tau}),v^{\varepsilon,m}_{\tau}\big)_S\\
        &=(\kappa,v^{\varepsilon,m})-\mu (I_{\hat{h}}(v^{\varepsilon,m}),v^{\varepsilon,m})\\
        &=(\kappa,v^{\varepsilon,m})+\mu (v^{\varepsilon,m}-I_{\hat{h}}(v^{\varepsilon,m}),v^{\varepsilon,m})-\mu||v^{\varepsilon,m}||^2\\
        &\leq \frac{1}{2\mu}||\kappa||^2+\frac{\mu}{2}||v^{\varepsilon,m}||^2+\frac{\mu}{2}||v^{\varepsilon,m}-I_{\hat{h}}(v^{\varepsilon,m})||^2-\frac{\mu}{2}||v^{\varepsilon,m}||^2\\
        &\leq \frac{||\kappa||^2}{2\mu}+\frac{\mu c_0{\hat{h}}^2}{2}||v^{\varepsilon,m}||_V^2.
    \end{aligned}\]
    It's obvious that $\big(\beta(v^{\varepsilon,m}_{\tau}),v^{\varepsilon,m}_{\tau}\big)_S\geq 0$. Noting the assumption $\mu c_0 {\hat{h}}^2\leq \nu$ and (\ref{kappa}), we can obtain
\[\frac{\mathrm{d}}{\mathrm{d}t}||v^{\varepsilon,m}||^2+\nu||v^{\varepsilon,m}||_V^2<\frac{M^2}{\mu},\]
which follows by
\begin{equation}\label{31}
    ||v^{\varepsilon,m}||^2\leq ||v^{\varepsilon}_0||^2+\frac{M^2T}{\mu} := \rho_H.
\end{equation}
Integrating $$\frac{\mathrm{d}}{\mathrm{d}t}||v^{\varepsilon,m}||+\nu||v^{\varepsilon,m}||_V^2<\frac{||k||^2}{\mu}$$ with respect to $t$, we have
\[\nu\int_{0}^{t}||v^{\varepsilon,m}(\varsigma )||_V\mathrm{d}\varsigma+||v^{\varepsilon,m}||^2-||v^{\varepsilon}_0||^2\leq \frac{M^2}{\mu}t,\]
and thus
\begin{equation}\label{32}
    \int_{0}^{t}||v^{\varepsilon,m}(\varsigma)||_V\mathrm{d}\varsigma\leq\frac{1}{\nu}||v^{\varepsilon}_0||^2+\frac{M^2T}{\nu\mu}:=\varrho_V.
\end{equation}

Then, taking the inner product of the first equation in (\ref{27}) with $Av^{\varepsilon,m}$, we obtain
\begin{equation}\label{33}
    \big(\frac{\mathrm{d}v^{\varepsilon,m}}{\mathrm{d}t},Av^{\varepsilon,m}\big)+\nu||Av^{\varepsilon,m}||^2+(B(v^{\varepsilon,m},v^{\varepsilon,m}),Av^{\varepsilon,m}) = (\kappa,Av^{\varepsilon,m})-\mu\big(P_{\kappa }(I_{\hat{h}}(v^{\varepsilon,m})),Av^{\varepsilon,m}\big).
\end{equation}
According to the result (\ref{a3}), we can deduced that
\[\begin{aligned}
    \big( \frac{\mathrm{d}v^{\varepsilon,m}}{\mathrm{d}t},Av^{\varepsilon,m} \big)&=\frac{1}{2}\frac{\mathrm{d}}{\mathrm{d}t}||v^{\varepsilon,m}||_V^2+\frac{1}{\nu}\left(\beta(v^{\varepsilon,m}_{\tau}),\frac{\mathrm{d}}{\mathrm{d}t}v^{\varepsilon,m}_{\tau}\right)_S\\
    &=\frac{1}{2}\frac{\mathrm{d}}{\mathrm{d}t}||v^{\varepsilon,m}||_V^2+\frac{1}{\nu}\frac{\mathrm{d}}{\mathrm{d}t} j_{\varepsilon}(v_{\tau}^{\varepsilon,m}),
\end{aligned}\]
which combining with (\ref{33}) yields
\begin{equation}\label{34}
    \begin{aligned}
        \frac{1}{2}\frac{\mathrm{d}}{\mathrm{d}t}||v^{\varepsilon,m}||_V^2+\frac{1}{\nu}\frac{\mathrm{d}}{\mathrm{d}t} j_{\varepsilon}(v_{\tau}^{\varepsilon,m}) +\nu||Av^{\varepsilon,m}||^2+(B(v^{\varepsilon,m},v^{\varepsilon,m}),Av^{\varepsilon,m})\\
         = (\kappa,Av^{\varepsilon,m})-\mu\big(P_{\kappa }(I_{\hat{h}}(v^{\varepsilon,m})),Av^{\varepsilon,m}\big).
    \end{aligned}
\end{equation}
Thanks to the inequality (\ref{13}) and Theorem \ref{thmnew}, there hold
\[\begin{aligned}
    \left| \left( B(v^{\varepsilon,m},v^{\varepsilon,m}),Av^{\varepsilon,m} \right) \right|&\leq c_1||v^{\varepsilon,m}||^{\frac{1}{2}}||v^{\varepsilon,m}||_V||Av^{\varepsilon,m}||^{\frac{3}{2}}\\
    &\leq\frac{1}{4}\left[ c_1\frac{6^{3/4}}{\nu^{3/4}} ||v^{\varepsilon,m}||^{\frac{1}{2}}||v^{\varepsilon,m}||_V\right]^4 +\frac{3}{4}\left[ \frac{\nu^{3/4}}{6^{3/4}}||Av^{\varepsilon,m}||^{\frac{3}{2}} \right]^{\frac{4}{3}}\\
    &=\frac{54c_1^4}{\nu^3}||v^{\varepsilon,m}||^2||v^{\varepsilon,m}||_V^4+\frac{\nu}{8}||Av^{\varepsilon,m}||^2,\\
|(\kappa,Av^{\varepsilon,m})|&\leq||\kappa||\cdot ||Av^{\varepsilon,m}||\leq \frac{2}{\nu}||\kappa||^2+\frac{\nu}{8}||Av^{\varepsilon,m}||^2.
\end{aligned}\]
Using (\ref{a3}) and the assumption $\mu c_0\hat{h^2}\leq\nu$ we obtain
\[\begin{aligned}
    -\mu\left( P_{\kappa }(I_{\hat{h}}(v^{\varepsilon,m})),Av^{\varepsilon,m} \right)&=\mu\left(v^{\varepsilon,m}-P_{\kappa }(I_{\hat{h}}(v^{\varepsilon,m})),Av^{\varepsilon,m}\right)-\mu(v^{\varepsilon,m},Av^{\varepsilon,m})\\
    &\leq \mu||v^{\varepsilon,m}-I_{\hat{h}}(v^{\varepsilon,m})||\cdot||Av^{\varepsilon,m}||\\
    &\qquad\qquad-\mu\left( ||v^{\varepsilon,m}||_V^2 +\frac{1}{\nu}(\beta(v^{\varepsilon,m}_{\tau}),v^{\varepsilon,m}_\tau)\right)_S\\
    &\leq \frac{\mu^2 c_0 {\hat{h}}^2}{\nu}||v^{\varepsilon,m}||_V^2-\mu||v^{\varepsilon,m}||_V^2+\frac{\nu}{4}||Av^{\varepsilon,m}||^2\\
    &\leq\frac{\nu}{4}||Av^{\varepsilon,m}||^2.
\end{aligned}\]
Substituting these inequalities into (\ref{34}), there holds that
\begin{equation}\label{35}
\frac{\mathrm{d}}{\mathrm{d}t}||v^{\varepsilon,m}||_V^2+\frac{2}{\nu}\frac{\mathrm{d}}{\mathrm{d}t} j_{\varepsilon}(v_{\tau}^{\varepsilon,m})+\nu||Av^{\varepsilon,m}||-\frac{108c_1^4}{\nu^3}||v^{\varepsilon,m}||^2||v^{\varepsilon,m}||_V^4\leq\frac{4M^2}{\nu}.
\end{equation}
Letting
$$\psi(t)=\exp{\left(-\frac{108c_1^4}{\nu^3}\int_{0}^{t}||v^{\varepsilon,m}||^2||v^{\varepsilon,m}||_V^2\mathrm{d}\varsigma \right)},$$
 then multiplying both sides of the \ref{35} by $\psi(t)$ and integrating the resulted inequality from $0$ to $t$, we get
\begin{equation}\label{36T}
\int_{0}^{t}\frac{\mathrm{d}(\psi(\varsigma )||v^{\varepsilon,m}||_V^2)}{\mathrm{d}t}\mathrm{d}\varsigma+\frac{2}{\nu}\int_{0}^{t}\psi(\varsigma )\frac{\mathrm{d}}{\mathrm{d}\varsigma}j_{\varepsilon}(v_{\tau}^{\varepsilon,m})\mathrm{d}\varsigma\leq\frac{4M^2T}{\nu}.
\end{equation}\label{36}
For the second term in (\ref{36T}) it is valid, by utilizing the integral by part and noting $j_\varepsilon\geq 0$,
\[\begin{aligned}
    \int_{0}^{t}\psi(\varsigma)\frac{\mathrm{d}}{\mathrm{d}\varsigma}j_{\varepsilon}(v_{\tau}^{\varepsilon,m}) \mathrm{d}\varsigma =&\psi(t)j_{\varepsilon}(v_{\tau}^{\varepsilon,m}(t))-j_{\varepsilon}(v_{\tau}^{\varepsilon,m}(0)) \\
    &+ \int_{0}^{t} \frac{108c_1^4}{\nu^3}||v^{\varepsilon,m}(\varsigma)||^2||v^{\varepsilon,m}(\varsigma)||_V^2 \psi(\varsigma) j_{\varepsilon}(v_{\tau}^{\varepsilon,m}) \mathrm{d}\varsigma\\
    \geq&-j_{\varepsilon}(v_{\tau}^{\varepsilon,m}(0)),
\end{aligned}\]
which combining with (\ref{36T}) yields
\[ ||v^{\varepsilon,m}(t)||_V^2\psi(t)-||v^{\varepsilon,m}(0)||_V^2- \frac{2}{\nu}j_{\varepsilon}(v_{\tau}^{\varepsilon,m}(0)) \leq \frac{4M^2T}{\nu},\]
i.e.,
\begin{equation}\label{37}
    ||v^{\varepsilon,m}(t)||_V^2\leq \exp\left(\frac{108c_1^4}{\nu^3}\rho_H\rho_V T\right)\left(||v^{\varepsilon,m}(0)||_V^2+\frac{2}{\nu}j_{\varepsilon}(v_{\tau}^{\varepsilon,m}(0))+\frac{4M^2T}{\nu}\right):=\rho_D.
\end{equation}
On the other hand, integrating (\ref{35}) with respect to $t$, we obtain
\[
    \begin{aligned}
        ||v^{\varepsilon,m}(t)||_V^2-||v^{\varepsilon,m}(0)||_V^2+\frac{2}{\nu}(j_{\varepsilon}(v_{\tau}^{\varepsilon,m}(t))-j_{\varepsilon}(v_{\tau}^{\varepsilon,m}(0)))\\
        +\nu\int_{0}^{t}||Av^{\varepsilon,m}||^2\mathrm{d}\varsigma \leq \int_{0}^{t}\frac{108c_1^4}{\nu^3}||v^{\varepsilon,m}||^2||v^{\varepsilon,m}||_V^4\mathrm{d}\varsigma +\frac{4M^2T}{\nu},
    \end{aligned}
\]
that implies that
\begin{equation}\label{38}
    \int_{0}^{t}||Av^{\varepsilon,m}||^2\mathrm{d}\varsigma \leq \left(\frac{108c_1^4}{\nu^4}\rho_H^2\rho_V^2+\frac{4M^2}{\nu^2}\right)T+||v^{\varepsilon,m}(0)||_V^2+\frac{2}{\nu^2}j_{\varepsilon}(v_{\tau}^{\varepsilon,m}(0)):=\varrho_{\mathcal{D}}.
\end{equation}
Noting that $v^{\varepsilon,m}(0)$ depends on the given $v_0$, by taking the limit of $m$ in (\ref{31}), (\ref{32}), (\ref{37}) and (\ref{38}), we can prove the global existence.

Next, we show that such solution is unique and depends continuously on the initial data. If $v^{\varepsilon}_1$ and $v^{\varepsilon}_2$ are two solutions of (\ref{27}) that have the common up bound $M_1$, by setting $\delta = v^{\varepsilon}_1-v^{\varepsilon}_2$, then there holds
\[
    \frac{\mathrm{d}}{\mathrm{d}t}\delta +\nu A\delta + B(v_1^{\varepsilon},\delta) + B(\delta,v_2^{\varepsilon})=-\mu P_{\kappa }I_{\hat{h}}(\delta).
\]
Similarly, by taking the inner product of the above equation with $\delta$, we can obtain
\[
    ||\delta||^2\leq \ \exp\left(\frac{27c^4M_1}{2\nu^3\lambda_1}t\right)||\delta(0)||^2=0,
\]
by noting that $v^{\varepsilon}_1$ and $v^{\varepsilon}_2$ have the same initial data $P_mv_0$. Thus, it is valid $||\delta|| = 0$ which means the solution is unique.
\end{proof}

Next, we will prove that the solution of the continuous data assimilation problem (\ref{27})  converges exponentially in time to the solution of the original problem (\ref{23}).

\begin{theorem}\label{th3}
    Let  $u^{\varepsilon}$ be a solution of the incompressible two-dimensional Navier-Stokes equations (\ref{23}) with nonlinear slip boundary conditions, and $v^{\varepsilon}$ be the solution of (\ref{27}). If $\mu>c_2^2\left(||u_0||^2+2(1/\nu+\nu)G\right)$, $\hat{h}^2\leq\frac{\nu}{c_0c_2^2\left(||u_0||^2+2(1/\nu+\nu)G\right)}$ where $c_2$ is the constant appearing in (\ref{14}). Then, for $T^*\geq \frac{1}{2\nu^2}$, it holds for all $t>T^*$ that
    \[||w||^2\leq ||u_0-v_0||^2e^{1+\gamma/2}\ e^{-\frac{\gamma}{2T}(t-t_0)},\]
where
    \[\gamma=\liminf_{t\to\infty} \int_{t}^{t+T^*}\mu - \frac{c_2^2}{\nu}||u^{\varepsilon}(\varsigma)||_V^2\mathrm{d}\varsigma.\]
\end{theorem}
\begin{proof}
    Setting $w=u^\varepsilon-v^\varepsilon$, subtracting (\ref{27}) from (\ref{23}) yields
    \begin{equation}\label{40}
        \frac{\mathrm{d}w}{\mathrm{d}t}+\nu Aw+B(u^{\varepsilon},w)+B(w,u^{\varepsilon})-B(w,w) = -\mu P_{\kappa }I_{\hat{h}}(w).
    \end{equation}
    Taking the inner product of (\ref{40}) with $w$ and using (\ref{12}), we obtain
    \begin{equation}\begin{aligned}\label{41T}
        \frac{1}{2}\frac{\mathrm{d}}{\mathrm{d}t}||w||^2+\nu(Aw,w)+(B(w,u^{\varepsilon}),w)=-\mu(P_{\kappa }I_{\hat{h}}(w),w).
    \end{aligned}\end{equation}
    Due to (\ref{a3}) and (\ref{14}), there hold
    \[\begin{aligned}
        \nu(Aw,w) = \nu ||w||_V^2 + (\beta(u_{\tau}^{\varepsilon})-\beta(v_{\tau}^{\varepsilon}),w_{\tau})_S,
    \end{aligned}\]
    \[
        |(B(w,u^{\varepsilon}),w)|\leq c_2||u^{\varepsilon}||_V ||w||_V||w||\leq \frac{c_2^2}{2\nu}||u^{\varepsilon}||_V^2 ||w||^2 +\frac{\nu}{2}||w||_V^2,
    \]
    and
    \[\begin{aligned}
        -\mu(P_{\kappa }I_{\hat{h}}(w),w) &= \mu (w-P_{\kappa }I_{\hat{h}}(w),w) - \mu ||w||^2\\
        &\leq \frac{\mu}{2}||w-I_{\hat{h}}(w)||^2 + \frac{\mu}{2}||w||^2 - \mu ||w||^2\\
        &\leq \frac{\mu c_0 {\hat{h}}^2}{2}||w||_V^2 - \frac{\mu}{2} ||w||^2\\
        &\leq \frac{\nu}{2}||w||_V^2 - \frac{\mu}{2} ||w||^2.
    \end{aligned}\]
where the assumption $\mu c_0\hat{h^2}\leq \nu$ is used in the last inequality. Putting these estimates into (\ref{41T}), we obtain that
    \begin{equation}\label{42}
        \frac{1}{2}\frac{\mathrm{d}}{\mathrm{d}t}||w||^2 + (\beta(u_{\tau}^{\varepsilon})-\beta(v_{\tau}^{\varepsilon}),w_{\tau})_S\leq \frac{c_2^2}{2\nu}||u^{\varepsilon}||_V^2 ||w|| - \frac{\mu}{2} ||w||^2.
    \end{equation}
     Due to the definition of $\beta(x)$, there holds
    \[\beta'(x) = g\frac{\varepsilon^2}{x^2+\varepsilon^2}\geq 0,\]
    which means that $\beta(x)$ an increasing function. Not lose of generation, assume $u_{\tau}^{\varepsilon} \geq v_{\tau}^{\varepsilon}$, then there holds $\beta(u_{\tau}^{\varepsilon}) \geq \beta(v_{\tau}^{\varepsilon})$. Consequently, it is valid that
    \[
        (\beta(u_{\tau}^{\varepsilon})-\beta(v_{\tau}^{\varepsilon}))(u_{\tau}^{\varepsilon}-v_{\tau}^{\varepsilon})\geq 0,
    \]
   which follows by
   \begin{equation}\label{43}
   (\beta(u_{\tau}^{\varepsilon})-\beta(v_{\tau}^{\varepsilon}),w_{\tau})_S\geq 0.
   \end{equation}
    Putting this estimate into (\ref{42}), we get
    \begin{equation}\label{44}
        \frac{\mathrm{d}}{\mathrm{d}t}||w||^2 +\left(\mu - \frac{c_2^2}{\nu}||u^{\varepsilon}||_V^2\right) ||w||^2\leq 0.
    \end{equation}
    Denoting
    \[\alpha(t):= \mu - \frac{c_2^2}{\nu}||u^{\varepsilon}(t)||_V^2,\]
    using (\ref{24b}) in Theorem 1 and noting the arbitrariness of $T^*$,  we have
    \[\int_{t}^{t+T^*}\alpha(\varsigma)\mathrm{d}\varsigma\geq \mu T^*-\frac{c_2^2}{\nu}\left(\frac{||u_0||^2}{2\nu}+(1+2T)\nu G^2\right).\]
    Therefore, if set $T^*=\frac{1}{2\nu^2}$ and $\mu>c_2^2\left(||u_0||^2+2(1/\nu+\nu)G\right)$, there holds
    \[\limsup_{t\to\infty}\int_{t}^{t+T^*}\alpha(\varsigma)\mathrm{d}\varsigma\geq 0.\]
    Substituting this into (\ref{44}) and using Lemma 1,  we obtain that
    \[||w||^2\leq ||u_0-v_0||^2e^{1+\gamma/2}\ e^{-\frac{\gamma}{2T}(t-t_0)}.\]
\end{proof}
Under the conditions of Theorem 4, we can easily conclude that, if the distance of the observation is sufficiently small, i.e., $\hat{h}^2\leq\frac{\nu}{c_0c_2^2\left(||u_0||^2+2(1/\nu+\nu)G\right)}$, then $||u^{\varepsilon}-v^{\varepsilon}||\to 0$  exponentially converge as $t\to\infty$.
\subsection{Problem with viscosity coefficient missing }
In this section, we consider the situation in which the viscosity coefficient is missing, but the initial data is given. Suppose that the true viscosity $\nu$ is unknown. To simulate the equations, we use a ``guessed" viscosity $\tilde{\nu}>0$ to replace $\nu$, then the CDA of this problem is as follows:
\begin{equation}\label{45}
\left\{\begin{aligned}
    &\frac{\mathrm{d} v^{\varepsilon,\nu} }{\mathrm{d}t}+ \tilde{\nu}  A  v^{\varepsilon,\nu} +B( v^{\varepsilon,\nu} , v^{\varepsilon,\nu} ) = f+\mu P_{\kappa }(I_{\hat{h}}(u^{\varepsilon})-I_{\hat{h}}( v^{\varepsilon,\nu} ))\quad&\mathrm{in} \ \Omega\times(0,T],\\
    & v^{\varepsilon,\nu} (0)=u_0\quad &\mathrm{in} \ \ \Omega,\\
    & v^{\varepsilon,\nu} =0,&\mathrm{on} \ \Gamma\times[0,T],\\
    & v^{\varepsilon,\nu} \cdot n =0,\ -2D_\tau( v^{\varepsilon,\nu} )=\beta( v^{\varepsilon,\nu} _{\tau})\quad &\mathrm{on} \ S\times[0,T].
\end{aligned}\right.
\end{equation}

Next, we will derive the convergence with respect to the viscosity coefficient for the CDA problem (\ref{45}). Then, based on the convergence analysis, we can recovery the true viscosity $\nu$, which make possible to simulate of the problem with viscosity coefficient missing.
\begin{theorem}\label{pr}
    Let $u^{\varepsilon}$ be a solution of the incompressible two-dimensional Navier-Stokes equations (\ref{23}) with nonlinear slip boundary conditions, and $ v^{\varepsilon,\nu} $ be the solution of (\ref{45}). Suppose  $\mu> c_2^2\left(||u_0||^2+(2\nu\tilde{\nu}+4)G\right)$ and  $\hat{h}^2\leq \frac{ \tilde{\nu} }{4c_0c_2^2\left(||u_0||^2+2(\nu+2/\tilde{\nu})G\right)}$ where $c_2$ is the constant appearing in (\ref{14}). Then for $T^* = \frac{1}{ \nu\tilde{\nu} }$, it holds for all $t>T^*$ that
\begin{equation}\label{46}
    ||u^{\varepsilon}(t)- v^{\varepsilon,\nu} (t)||^2\leq
     c_3\cdot\frac{( \tilde{\nu} - \nu )^2}{ \nu  \tilde{\nu} },
\end{equation}
where
\[
    c_3:=\sup_{t>t_0}\int_{t}^{t+T^*}4c_2^2||g||_V^2+ \nu ||u^{\varepsilon}||_V^2\mathrm{d}\varsigma
\]
and
\[\gamma:=\liminf_{t\to \infty}\int_{t}^{t+T^*}\left(\mu - \frac{2c_2^2}{ \tilde{\nu} }||u^{\varepsilon}(\varsigma)||_V^2\right)\mathrm{d}\varsigma>0.\]
\end{theorem}
\begin{proof}
  Denoting $w^\nu=u^{\varepsilon}- v^{\varepsilon,\nu} $, and subtracting  (\ref{45}) from (\ref{23}), taking inner product of its first equation with $w^\nu$ and using (\ref{12})  we obtain
    \begin{equation}\label{47T}\begin{aligned}
        \frac{1}{2}\frac{\mathrm{d}}{\mathrm{d}t}||w^\nu||^2+\tilde{\nu}(A w^\nu,w^\nu)+(\nu-\tilde{\nu})(Au^\varepsilon,w^\nu)\\
        =-(B(w^\nu,u^{\varepsilon}),w^\nu)-\mu (P_{\kappa }(I_{\hat{h}}(w^\nu)),w^\nu).
    \end{aligned}\end{equation}
    Thanks to (\ref{a3}), it's valid that
    \begin{align*}
        &(\tilde{\nu}Aw^\nu,w^\nu)=\tilde{nu}||w^\nu||_V^2+\left(\beta(w_\tau^\nu),w_\tau^\nu\right)_S\geq \tilde{\nu} ||w^\nu||_V^2\\
        &\left((\nu-\tilde{\nu})Au^\varepsilon,w^\nu\right)=(\nu-\tilde{\nu})\int_{\Omega}\nabla u^\varepsilon\nabla w^\nu\mathrm{d}\textbf{x}+\left(\beta(u_\tau^\varepsilon)-\frac{\tilde{\nu}}{\nu}\beta(v_\tau^{\varepsilon,\nu}),w_\tau^\nu\right)_S.
    \end{align*}
    Due to (\ref{43}), there holds
    \[\begin{aligned}
        \left(\beta(u_\tau^\varepsilon)-\frac{\tilde{\nu}}{\nu}\beta(v_\tau^{\varepsilon,\nu}),w_\tau^\nu\right)_S=& \left(\beta(u_\tau^\varepsilon)-\frac{\tilde{\nu}}{\nu}\beta(u_\tau^\varepsilon)+\frac{\tilde{\nu}}{\nu}\beta(u_\tau^\varepsilon)-\frac{\tilde{\nu}}{\nu}\beta(v_\tau^{\varepsilon,\nu}),w_\tau^\nu\right)_S\\
        \geq &\frac{\nu-\tilde{\nu}}{\nu}\left(\beta(u_\tau^\varepsilon),w_\tau^\nu\right)_S
    \end{aligned}\]
    Noting that
    \begin{align*}
        \frac{ \tilde{\nu} - \nu }{ \nu }(\beta(u_{\tau}^{\varepsilon}),w^\nu_{\tau})_S\leq &\frac{| \tilde{\nu} - \nu |}{ \nu }\int_S |gw^\nu_{\tau}|\mathrm{d}s\\
        \leq& \frac{| \tilde{\nu} - \nu |}{ \nu }||g||_{\partial\Omega}||w^\nu||_{\partial\Omega}\\
        \leq& c_4\frac{| \tilde{\nu} - \nu |}{ \nu }||g||_V||w^\nu||_V,
    \end{align*}
    where $c_4$ is a constant generated from the trace theorem. Utilizing the Cauchy-Schwarz and Young's inequality, we obtain
    \[\begin{aligned}
        c_4\frac{| \tilde{\nu} - \nu |}{ \nu }||g||_V||w^\nu||_V \leq&c_4\frac{| \tilde{\nu} - \nu |}{ \nu }(\frac{4c_4| \tilde{\nu} - \nu |}{2 \tilde{\nu} }||g||_V^2+\frac{ \tilde{\nu} }{2\times4c_4| \tilde{\nu} - \nu |}||w^\nu||_V^2)\\
        =&\frac{2c_4^2( \tilde{\nu} - \nu )^2}{ \nu  \tilde{\nu} }||g||_V^2+\frac{ \tilde{\nu} }{8}||w^\nu||_V^2,\\
        ( \tilde{\nu} - \nu )\int_{\Omega}{\nabla u^{\varepsilon}\nabla w^\nu}\mathrm{d}\textbf{x} \leq& | \tilde{\nu} - \nu |\cdot ||u^{\varepsilon}||_V||w^\nu||_V\\
       \leq & \frac{| \tilde{\nu} - \nu |^2}{2 \tilde{\nu} }||u^{\varepsilon}||_V^2+\frac{ \tilde{\nu} }{2}||w^\nu||_V^2,
    \end{aligned}\]
    and by inequality (\ref{14}) and (\ref{25})
    \[\begin{aligned}
        &-(B(w^\nu,u^{\varepsilon}),w^\nu)-\mu (P_{\kappa }(I_{\hat{h}}(w^\nu)),w^\nu)\\
        \leq&c_2||w^\nu||\cdot||w^\nu||_V||u^{\varepsilon}||_V-\mu (P_{\kappa }(I_{\hat{h}}(w^\nu)-w^\nu),w^\nu)-\mu||w^\nu||^2\\
        \leq&c_2||w^\nu||\cdot||w^\nu||_V||u^{\varepsilon}||_V+\mu \sqrt{c_0{\hat{h}}^2}||w^\nu||_v||w^\nu||-\mu||w^\nu||^2\\
        \leq&\frac{c_2^2}{ \tilde{\nu} }||w^\nu||^2||u^{\varepsilon}||_V^2+\frac{ \tilde{\nu} }{4}||w^\nu||_V^2+\frac{\mu c_0{\hat{h}}^2}{2}||w^\nu||_V^2-\frac{\mu}{2}||w^\nu||^2.
    \end{aligned}\]
    Putting these estimates into (\ref{47T}), we can derive
    \begin{equation}\label{49}
    \begin{aligned}
        &\frac{\mathrm{d}}{\mathrm{d}t}||w^\nu||^2+(\frac{ \tilde{\nu} }{4}-\mu c_0{\hat{h}}^2)||w^\nu||_V^2+(\mu-\frac{2c_2^2}{ \tilde{\nu} }||u^{\varepsilon}||_V)||w^\nu||^2\\
        \leq& \frac{( \tilde{\nu} - \nu )^2}{ \nu  \tilde{\nu} }(4c_4^2||g||_V^2+ \nu ||u^{\varepsilon}||_V^2).
    \end{aligned}
    \end{equation}
    Thanks to the assumptions on $\mu$ and ${\hat{h}}$, it can be deduced that
    \begin{equation}\label{50}
        \frac{\mathrm{d}}{\mathrm{d}t}||w^\nu||^2+(\mu-\frac{2c_2^2}{ \tilde{\nu} }||u^{\varepsilon}||_V)||w^\nu||^2\leq \frac{( \tilde{\nu} - \nu )^2}{ \nu \tilde{\nu} }(4c_42^2||g||_V^2+ \nu ||u^{\varepsilon}||_V^2).
    \end{equation}
    Noting  that $\mu> c_2^2\left(||u_0||^2+2(\nu+2/\tilde{\nu})G\right)$, for a fixed $T^*=\frac{1}{ \nu \tilde{\nu}}$, we have
    \[\begin{aligned}
        \gamma&:=\liminf_{t\to\infty}\int_{t}^{t+T}\mu-\frac{2c_2^2}{ \tilde{\nu} }||u^{\varepsilon}(\varsigma)||_V^2\mathrm{d}\varsigma\\
        &\geq \mu T^*-\frac{2c_2^2}{ \tilde{\nu} }(\frac{||u_0||^2}{2\nu}+(1+2T)G)\\
        &>0.
    \end{aligned}
    \]
    Obviously, it is valid  that
    \[\Lambda:=\liminf_{t\to \infty}\int_{t}^{t+T^*}\alpha^-(\varsigma)\mathrm{d}\varsigma=0,\]
    for large enough $t>t_0$.
    By applying Lemma~\ref{lem1} and using the given initial data $u_0$, we obtain that for all $t > t_0$,
\begin{align*}
    \|w^\nu(t)\|^2
    &\;\le\;
    \|w^\nu(t_0)\|^2 \,
    e^{\,1+\frac{\gamma}{2}}
    \,e^{-\frac{\gamma}{2T}(t-t_0)}
    \;+\;
    c_3 \cdot \frac{(\tilde{\nu} - \nu)^2}{\nu \,\tilde{\nu}}
    \\
    &\;=\;
    c_3 \cdot \frac{(\tilde{\nu} - \nu)^2}{\nu \,\tilde{\nu}},
\end{align*}
where
\[
    c_3
    \;=\;
    \sup_{t > t_0}
    \int_{t}^{t+T^*}
      \bigl(4c_2^2 \,\|g\|_V^2 \;+\; \nu\,\|u^{\varepsilon}\|_V^2\bigr)
    \,\mathrm{d}\varsigma
\]
is a bounded constant.

\end{proof}
We can recover the unknown parameter by using the result in Theorem 4, the specific see \cite{ref15}, and combining theorem 3 and theorem 4, we can get the error that both the initial data and the parameter are unknown.
\subsection{Finite Element Approximation}
In this subsection, we will consider the  finite element approximation  for the continuous data assimilation equation (\ref{45}).  Let $\mathcal{T}_h=\{K\}$ be a family of regular triangular partition with the diameter $0<h<1$,   $P_r(K)$ be the space of the polynomials on $K$ with degree at most $r$. Define the finite element subspaces of $V_h$ and $M_h$ by
\[
    V_h =\left\{ v \in V: v|_K \in P_2(K)^2, \forall K \in \mathcal{T}_h\right\},
\]
\[
    M_h =\left\{ q \in M: Q|_K \in P_1(K), \forall K \in \mathcal{T}_h\right\}.
\]
Then the  finite element approximations of (\ref{26}) and (\ref{45}) are as follows, respectively: find $(v_h^{\varepsilon}, p_h^{\varepsilon})\in V_h\times M_h$ such that
\begin{equation}\label{n52}
\left\{\begin{aligned}
    &(v_{ht}^{\varepsilon},w_h) + \nu a({ v_h^{\varepsilon}  },w_h) + b({ v_h^{\varepsilon}  },{ v_h^{\varepsilon}  },w_h) - d(w_h,p_h^{\varepsilon}) +  (\beta({ v_{h\tau}^{\varepsilon}  }),{w_{h\tau}})_S\\
    &\qquad+\mu(I_{\hat{h}}( v_h^{\varepsilon}  )-I_{\hat{h}}(u_h^{\varepsilon}),w_h)= (f,w_h), \quad \forall w_h\in V_h.\\
    &d({ v_h^{\varepsilon}  },q_h)= 0,\quad \forall q_h\in M_h,
\end{aligned}\right.
\end{equation}
and find $(v_h^{\varepsilon,\nu}, p_h^{\varepsilon,\nu})\in V_h\times M_h$ such that
\begin{equation}\label{n151}
\left\{\begin{aligned}
    &(v_{ht}^{\varepsilon,\nu},w_h) + \tilde{\nu} a({ v_h^{\varepsilon,\nu}  },w_h) + b({ v_h^{\varepsilon,\nu}  },{ v_h^{\varepsilon,\nu}  },w_h) - d(w_h,p_h^{\varepsilon,\nu}) +  (\beta({ v_{h\tau}^{\varepsilon,\nu}  }),{w_{h\tau}})_S\\
    &\qquad+\mu(I_{\hat{h}}( v_h^{\varepsilon,\nu}  )-I_{\hat{h}}(u_h^{\varepsilon}),w_h)= (f,w_h), \quad \forall w_h\in V_h.\\
    &d({ v_h^{\varepsilon,\nu}  },q_h)= 0,\quad \forall q_h\in M_h.
\end{aligned}\right.
\end{equation}

Next, we deduce the error estimates for (\ref{n52}) and (\ref{n151}). For the regularization process, there holds the following estimate.
\begin{theorem} \label{th5}\cite{ref9,ref21}
Assume that $u_0\in H$, $f \in L^2(0,T,V^*)$, $g \in L^2(0,T,L^2(S))$ and $g \geq 0$. Let  $u$ and  $u^{\varepsilon}$ be the solutions (\ref{1})-(\ref{5}) and (\ref{23}), respectively. Then there holds for all $t\in(0, T]$ that
\[||u(t)-u^{\varepsilon}(t)||\leq c_r\sqrt{\varepsilon},\]
where $c_r$ is a  positive constant.
\end{theorem}

For the finite element approximation, there holds the following error estimate.
\begin{theorem}\label{th6}\cite{ref9}
    Assume that $f,f_t\in L^2(0,T,H),\ u_0\in V_{\sigma}$,$\int_{S}g\sqrt{{u_0}_{\tau}^2+\varepsilon^2}\mathrm{d}s\leq c$, $g\in H^{1/2}(S)$, $g\geq 0.$ and $f \in L^{\infty}(0,T,H)$, $u_0\in \mathcal{D}(A)$. Let $u^{\varepsilon}$ and $u^{\varepsilon}_h$ be the solutions of (\ref{23}) and its finite element approximation, respectively. Then there holds  for all $t\in(0, T]$ that
    \[||u^{\varepsilon}(t)-u^{\varepsilon}_h(t)||\leq c_f h,\]
     where $c_f > 0$ is a positive constant.
\end{theorem}

Then, we arrive at the main convergence theorems.
\begin{theorem}
    Under the assumption of Theorems \ref{th3}, \ref{th5} and \ref{th6}, let  $u$ and  $u_h^{\varepsilon}$ be the solutions (\ref{1})-(\ref{5}) and (\ref{n52}), then there holds  for all $t\in(0, T]$ that
\begin{equation}\label{n255}
    ||u(t)- v_h^{\varepsilon}  (t)|| \leq c_r\sqrt{\varepsilon}+c_f h+||u_0-v_0||e^{1/2+\gamma/4} e^{-\frac{\gamma}{4T}(t-t_0)}.
\end{equation}
\end{theorem}
\begin{proof}
The results can be easily deduced by using  Theorems \ref{th3}, \ref{th5} and \ref{th6} and the triangle inequality
\begin{equation}\label{53}
    ||u(t)- v_h^{\varepsilon}  (t)||\leq ||u(t)-u^{\varepsilon}(t)||+||u^{\varepsilon}(t)-u_h^{\varepsilon}(t)||+||u_h^{\varepsilon}(t)- v_h^{\varepsilon}  (t)||.
\end{equation}
\end{proof}

\begin{theorem}
    Under the assumption of Theorems \ref{pr}, \ref{th5} and \ref{th6}, let  $u$ and  $u_h^{\varepsilon,\nu}$ be the solutions (\ref{1})-(\ref{5}) and (\ref{n151}), then there holds  for all $t\in(0, T]$ that
\begin{equation}\label{54}
    ||u(t)- v_h^{\varepsilon,\nu}  (t)|| \leq c_r\sqrt{\varepsilon}+c_f h+c_3\cdot\frac{( \tilde{\nu} - \nu )}{\sqrt{ \nu  \tilde{\nu} }},
\end{equation}
where $c_3$ and $\gamma$ are constants appearing in Theorem \ref{pr}.
\end{theorem}
\begin{proof}
The results can be easily deduced by using  Theorems \ref{pr}, \ref{th5} and \ref{th6} and the triangle inequality
\begin{equation}\label{n355}
    ||u(t)- v_h^{\varepsilon,\nu}  (t)||\leq ||u(t)-u^{\varepsilon}(t)||+||u^{\varepsilon}(t)-u_h^{\varepsilon}(t)||+||u_h^{\varepsilon}(t)- v_h^{\varepsilon,\nu}  (t)||.
\end{equation}
\end{proof}

\subsection{Numerical Experiment}
In this section, we present several numerical examples to demonstrate the correctness and efficiency of the continuous data assimilation method developed above for solving the Navier-Stokes equations with nonlinear slip boundary conditions. In all simulations, we employ the Crank-Nicolson scheme with a time step  $\Delta t=0.001$ for the time discretization of the finite element approximation (59). The spatial mesh is set to be $h=1/32$, and the relaxation parameter is chosen to be $\mu=20$. Additionally, we initialize the velocity field with a guess $v_0^\varepsilon=0$, and utilize numerical solutions obtained via the finite element method with a mesh size of $\hat{h}=1/8$ for the ``observational data" in the feedback control term.

\begin{figure}
    \centering
    \caption*{FEM \qquad\qquad\qquad\qquad\qquad\qquad \quad CDA}
    \begin{minipage}{0.3\textwidth}
        \centering
        \includegraphics[width=\linewidth]{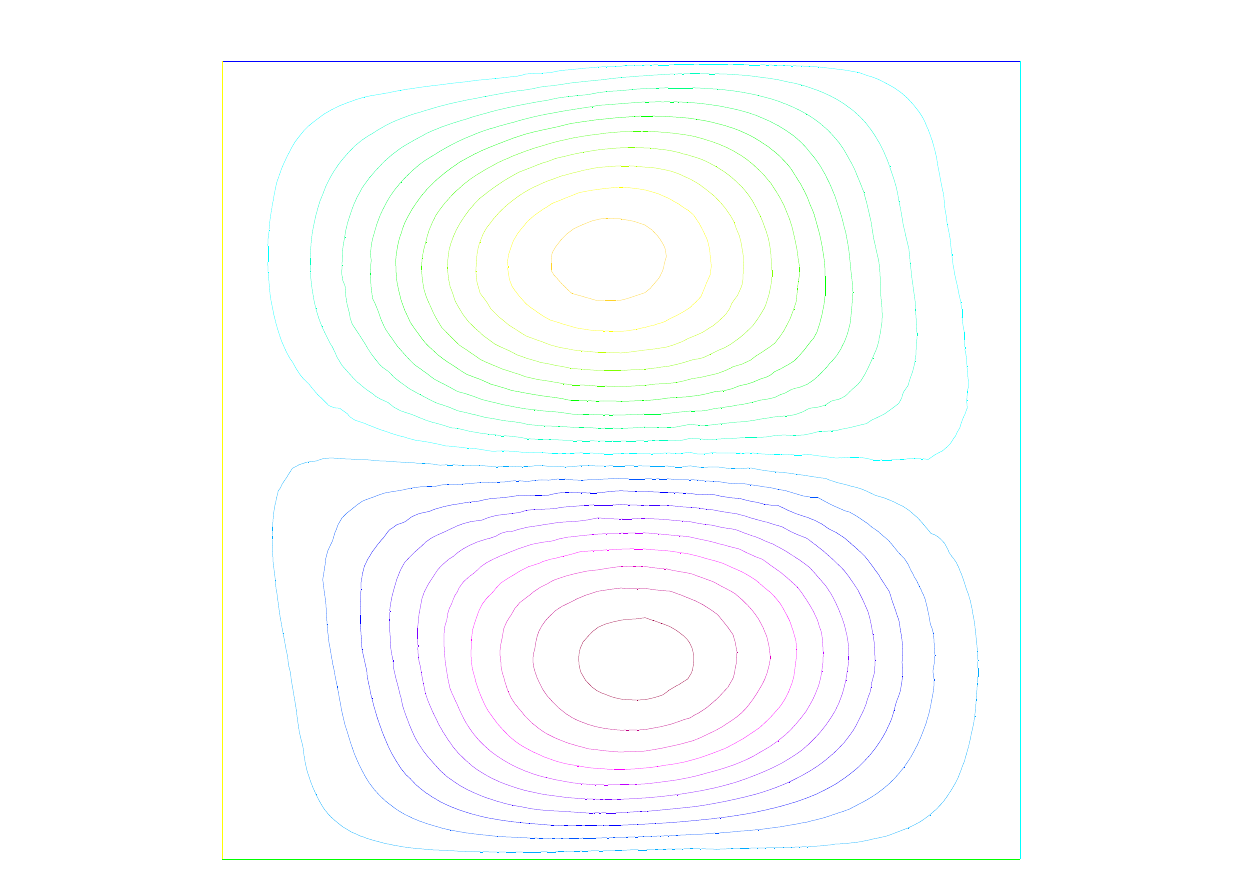}
    \end{minipage}%
    \begin{minipage}{0.3\textwidth}
      \centering
      \includegraphics[width=\linewidth]{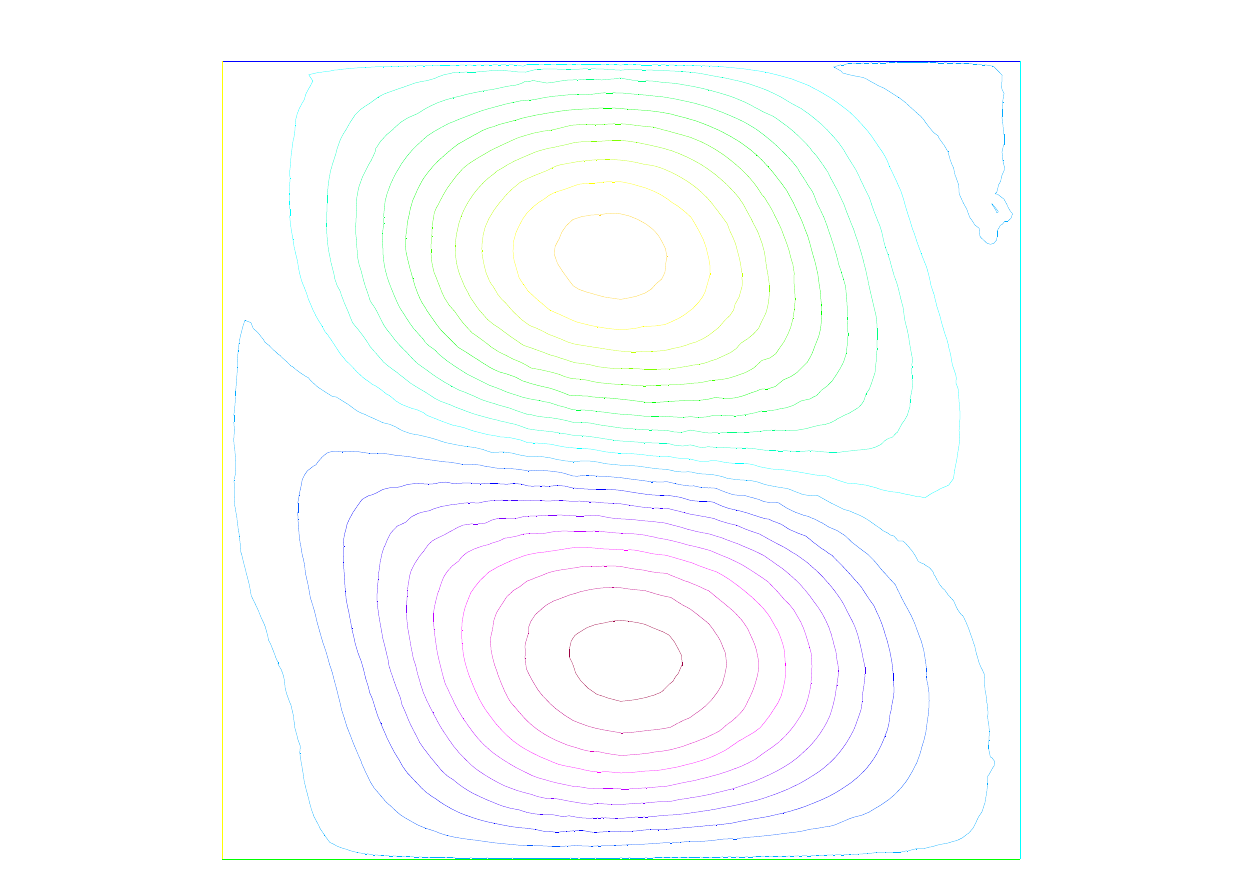}
    \end{minipage}%
    \vskip 0.2cm
    \caption*{$(a).~t=0.1$}

    \begin{minipage}{0.3\textwidth}
        \centering
        \includegraphics[width=\linewidth]{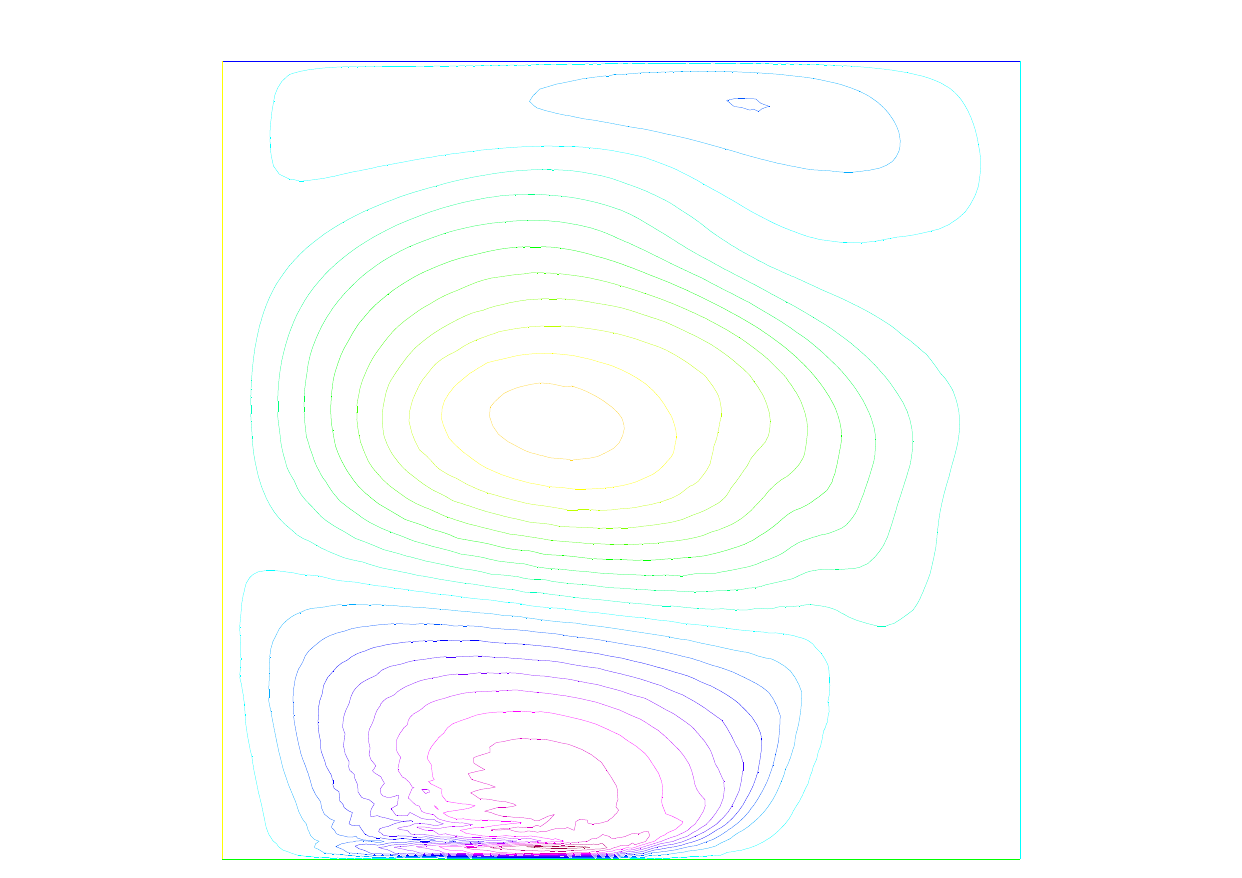}
    \end{minipage}%
    \begin{minipage}{0.3\textwidth}
      \centering
      \includegraphics[width=\linewidth]{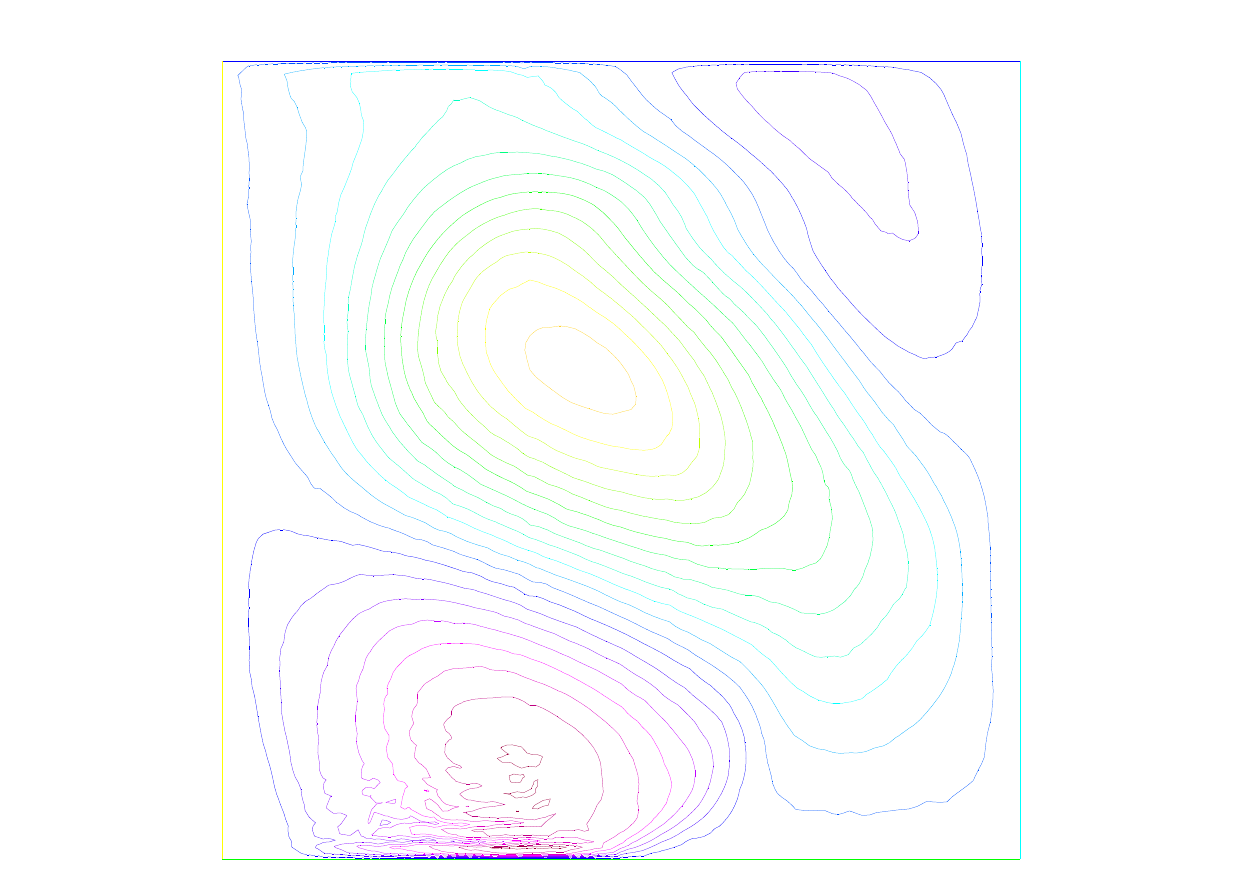}
    \end{minipage}%
    \vskip 0.2cm
    \caption*{$(b).~t=10$}

    \begin{minipage}{0.3\textwidth}
        \centering
        \includegraphics[width=\linewidth]{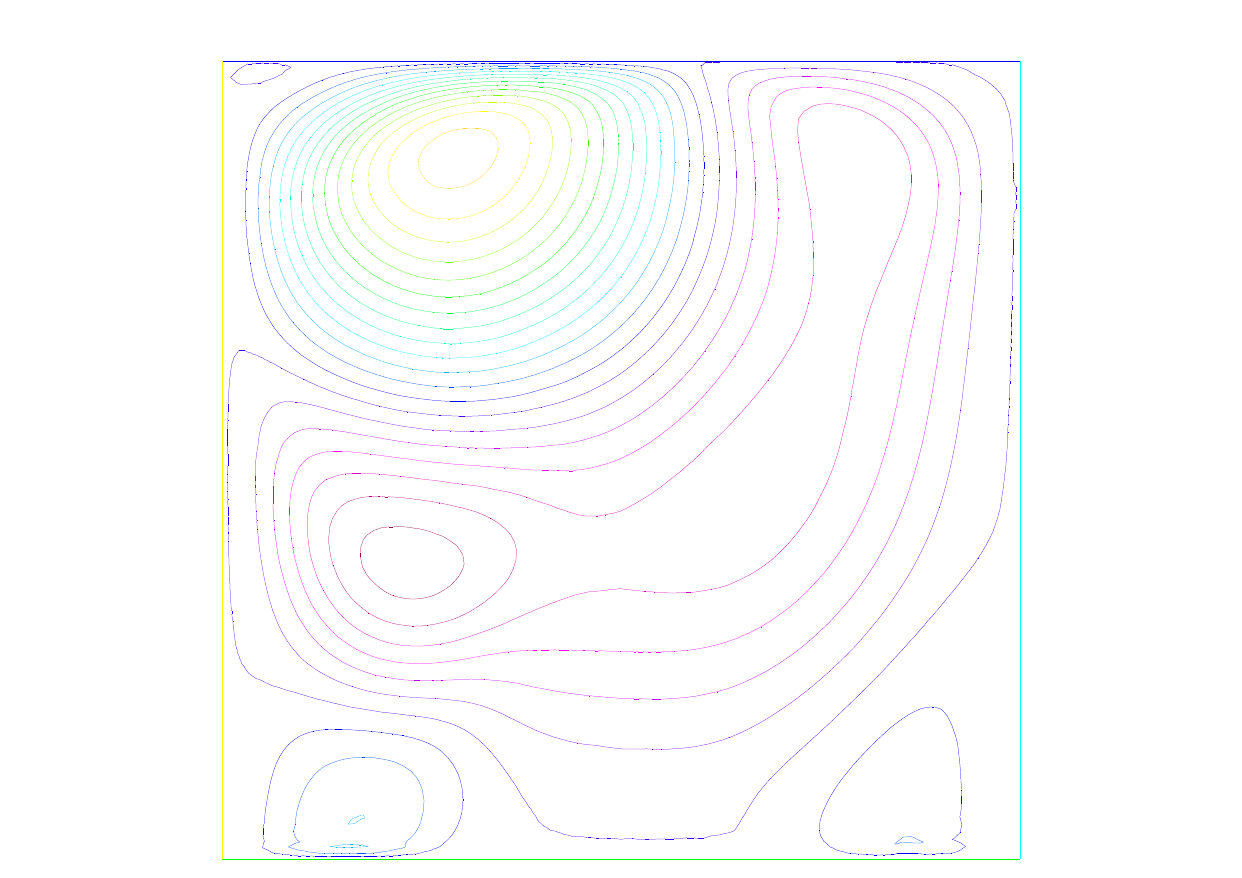}
    \end{minipage}
    \begin{minipage}{0.3\textwidth}
      \centering
      \includegraphics[width=\linewidth]{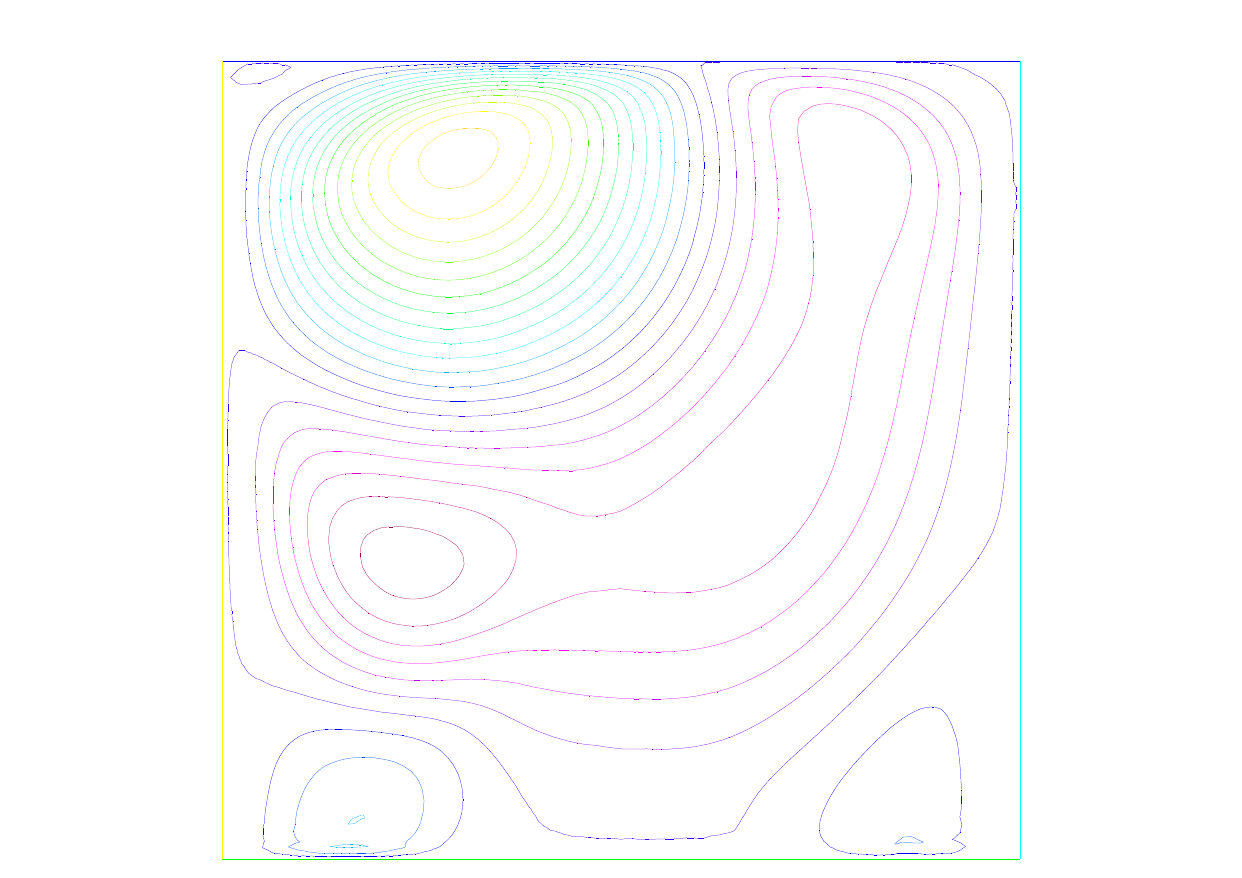}
    \end{minipage}
    \vskip 0.2cm
    \caption*{$(c).~t=30$}

    \caption{Comparisons of the simulations $u_{1h}^{\varepsilon}$ at different times (left: FEM, right: CDA).}

\end{figure}

\subsubsection{Flow in simple rectangular domain}
To validate the considered continuous data assimilation method, we first test the problem on the unit square $(0,1)^2$. The exact initial data $u_0$ is given by $$u_{10}(x,y)=2\pi \sin^2(\pi x)\sin (\pi y)\cos (\pi y),$$ $$u_{20}(x,y)=-2\pi\sin (\pi x)\cos (\pi x)\sin^2(\pi y),$$ which satisfies $\nabla\cdot u_0=0$. And the bottom boundary of the domain is set as the nonlinear slip boundary $S$, while the remaining boundaries are treated as homogeneous Dirichlet boundaries $\Gamma$. Letting the friction coefficient $g=1$, the regularization parameter $\varepsilon=10^{-5}$ and the viscosity coefficient $\nu=0.001$, we collect $u_1$ and $u_2$ generated from the continuous data assimilation with the initial data missing in Figures 1 and 2 where the initial guess value is set to 0, respectively. For comparison, we also present the corresponding numerical solutions got by the finite element method with the exact initial conditions.

\begin{figure}
    \centering
    \caption*{FEM \qquad\qquad\qquad\qquad\qquad\qquad \quad CDA}
    \begin{minipage}{0.3\textwidth}
        \centering
        \includegraphics[width=\linewidth]{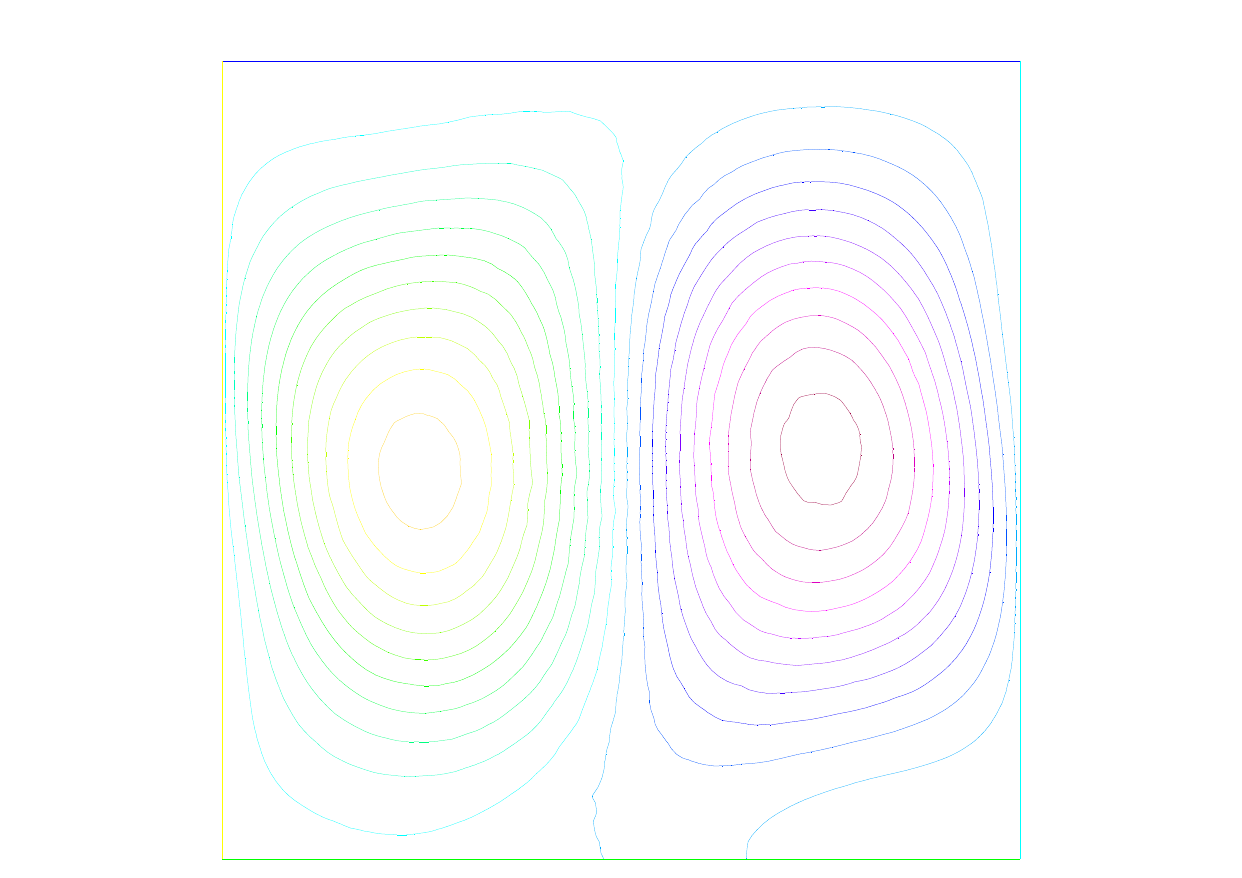}
    \end{minipage}%
    \begin{minipage}{0.3\textwidth}
        \centering
        \includegraphics[width=\linewidth]{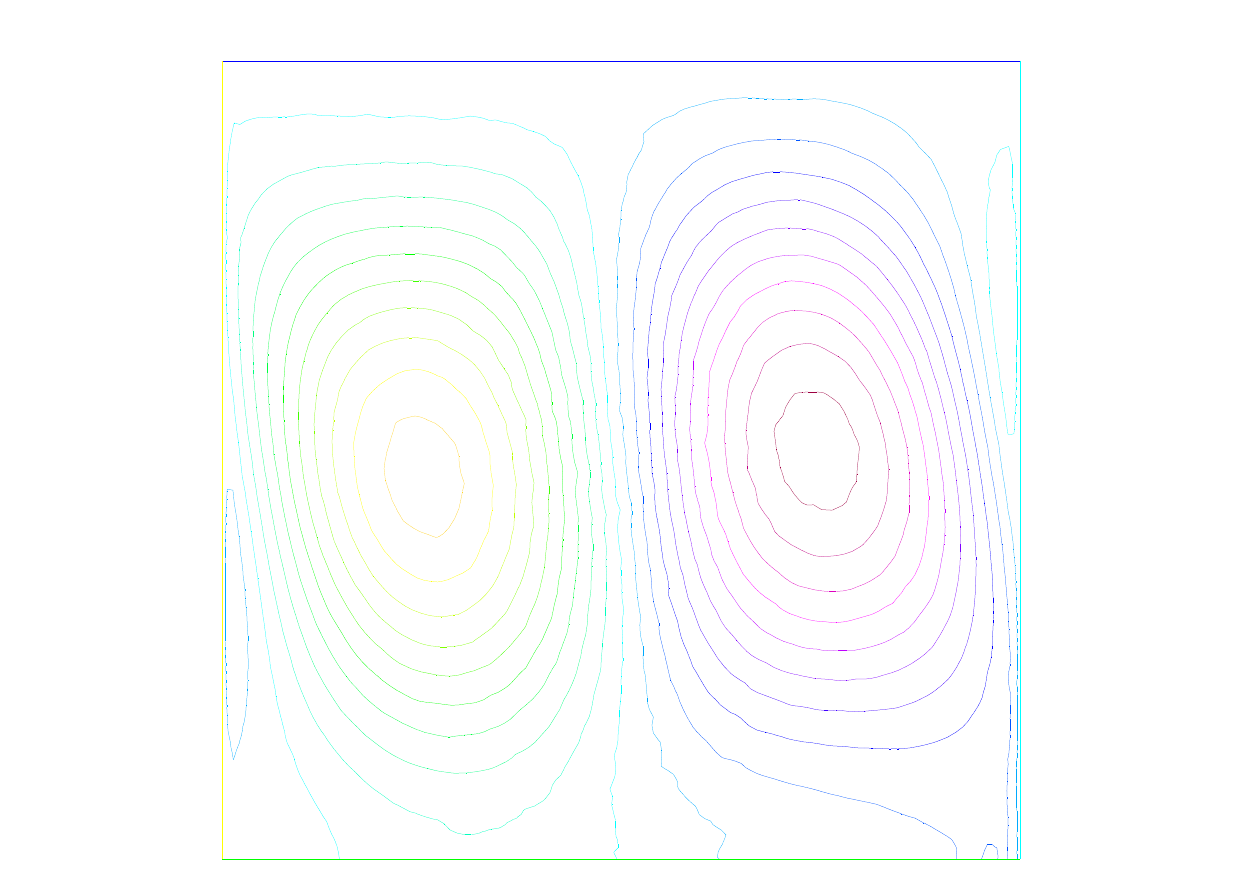}
    \end{minipage}%
    \vskip 0.2cm
    \caption*{$(a).~t=0.1$}

    \begin{minipage}{0.3\textwidth}
        \centering
        \includegraphics[width=\linewidth]{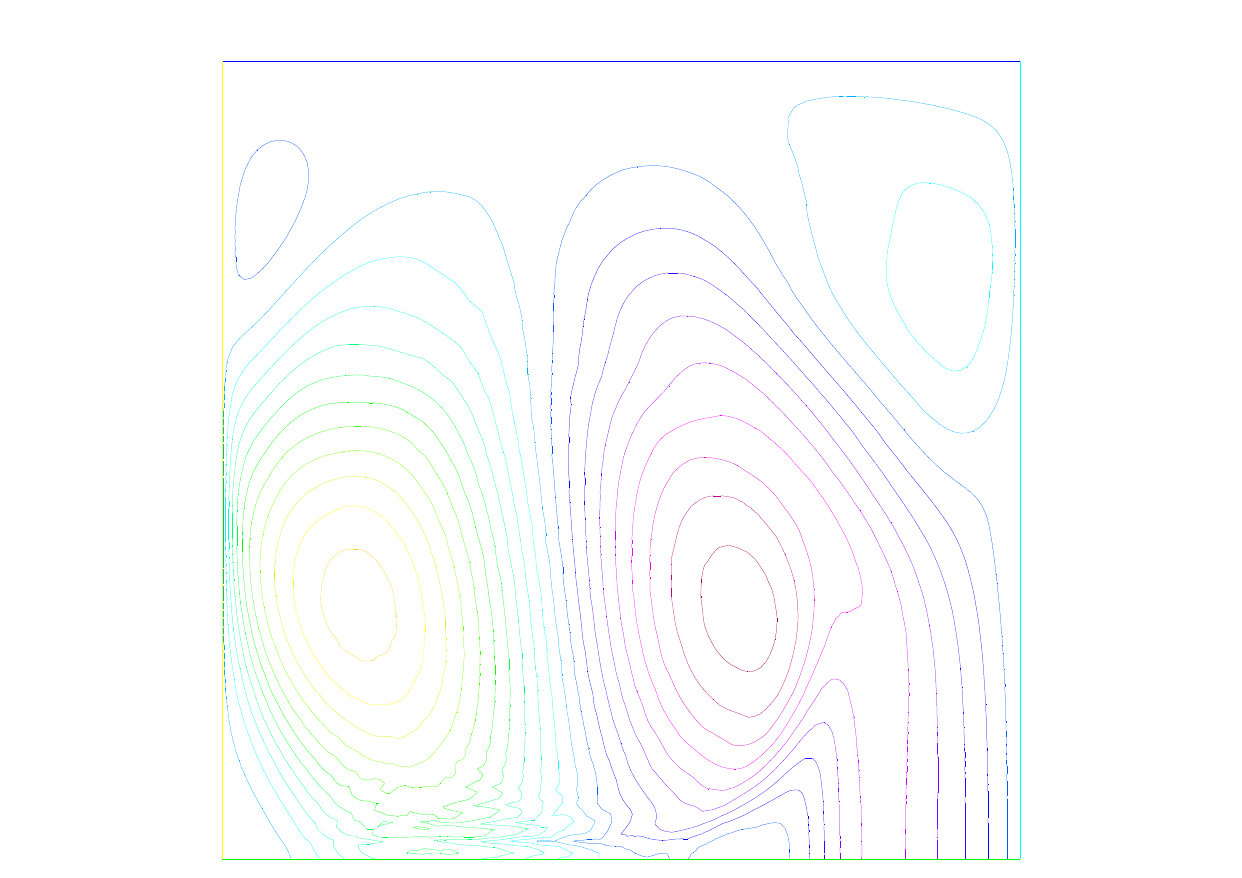}
    \end{minipage}%
    \begin{minipage}{0.3\textwidth}
        \centering
        \includegraphics[width=\linewidth]{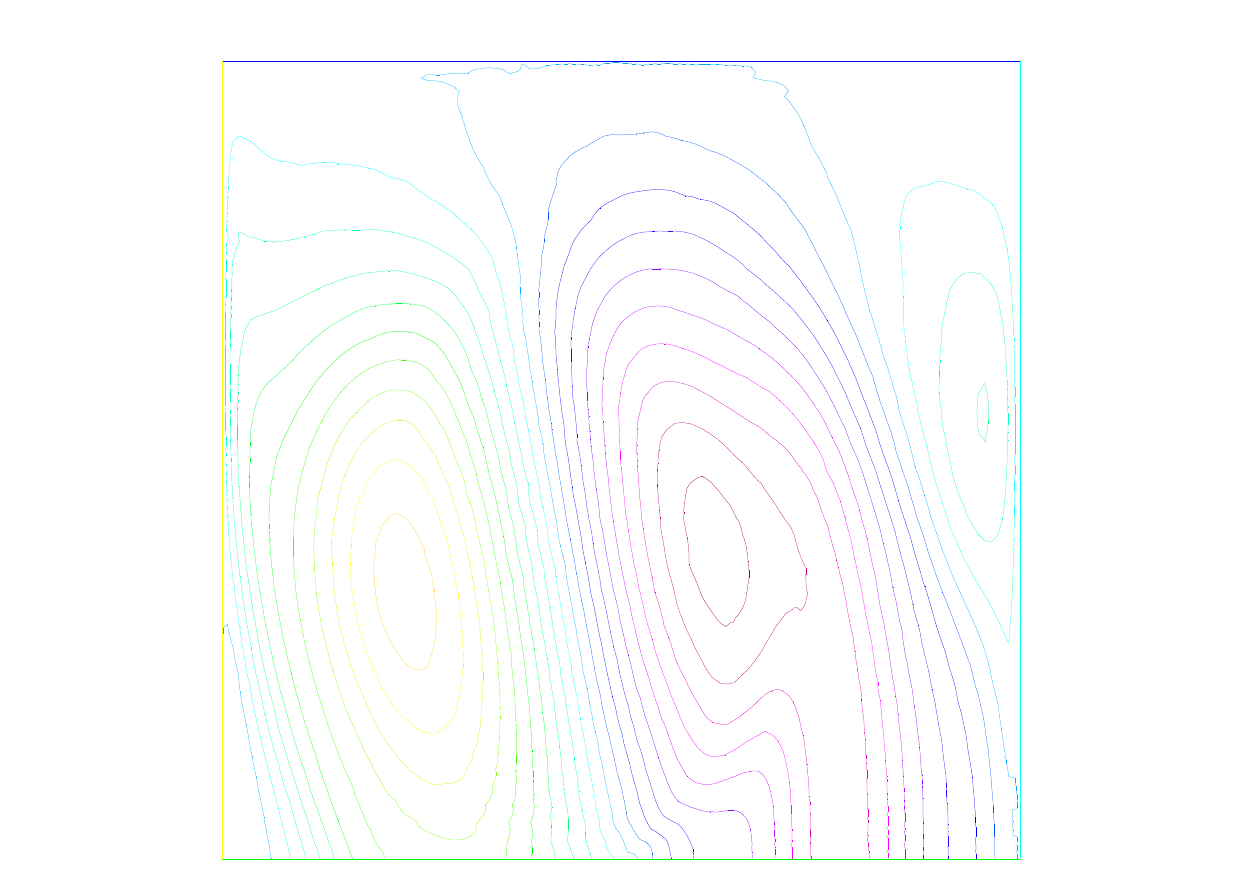}
    \end{minipage}%
    \vskip 0.2cm
    \caption*{$(b).~t=10$}

    \begin{minipage}{0.3\textwidth}
        \centering
        \includegraphics[width=\linewidth]{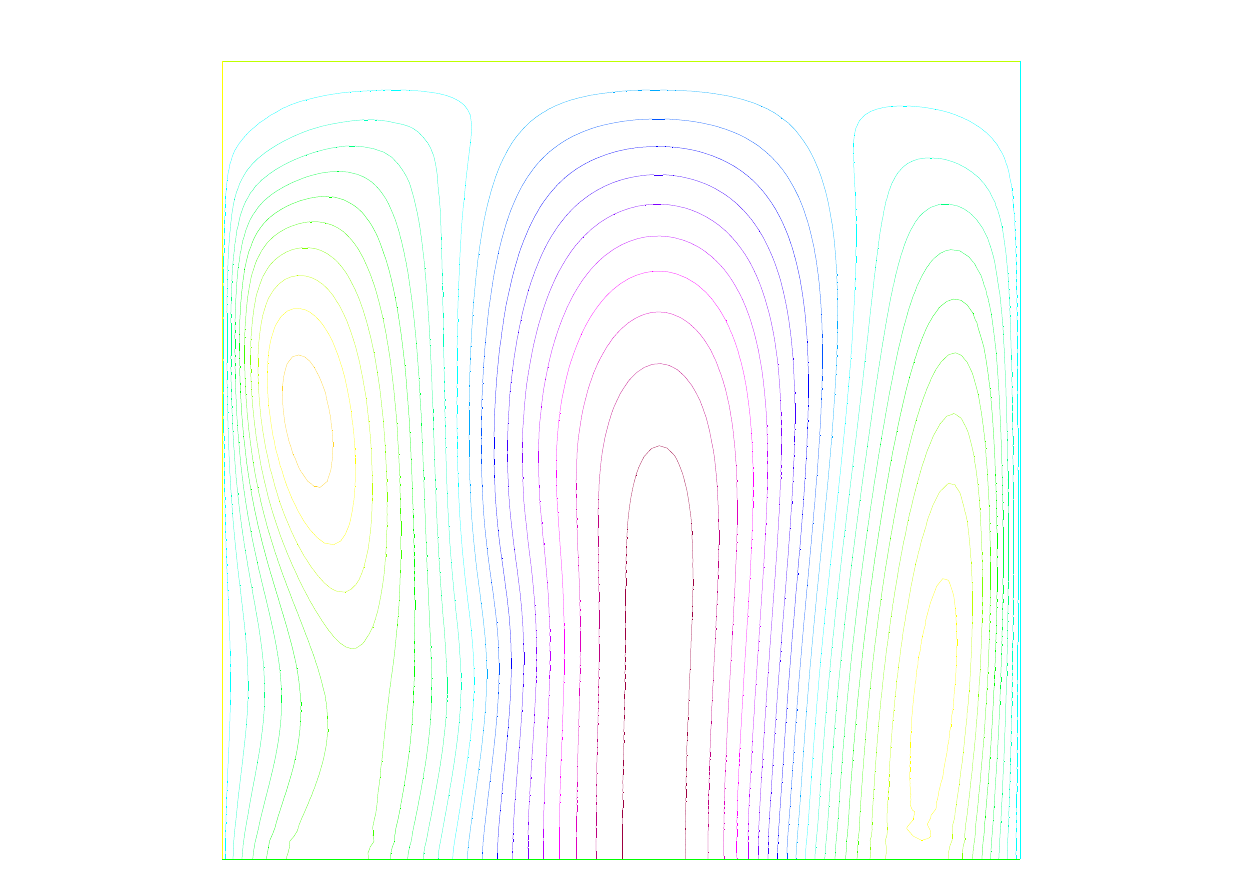}
    \end{minipage}
    \begin{minipage}{0.3\textwidth}
        \centering
        \includegraphics[width=\linewidth]{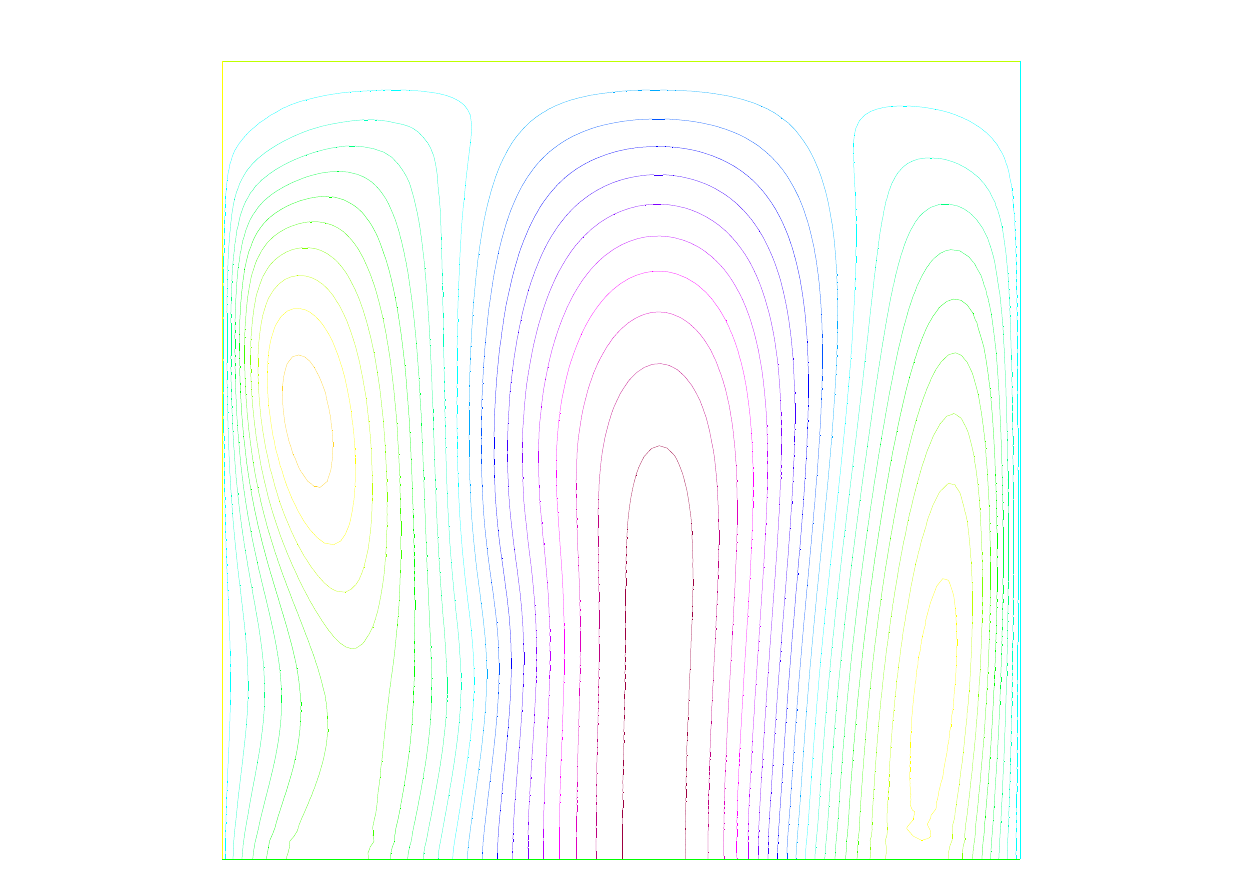}
    \end{minipage}
    \vskip 0.2cm
    \caption*{$(c).~t=30$}

    \caption{Comparisons of the simulations $u_{2h}^{\varepsilon}$ at different times (left: FEM, right: CDA).}

\end{figure}

From these figures we can see that the velocity contours from the two methods show noticeable differences at the beginning $(t=0.1)$. However, as the time develops, the results from the continuous data assimilation and the finite element method become increasingly similar. When the time $t=30$, the results are nearly indistinguishable. In addition, the errors between the continuous data assimilation with initial data missing (IDM) and the finite element solutions are summarized in Table 1 (see the second line), which clearly demonstrate the exponential convergence of the continuous data assimilation method.
Moreover, we investigate the error developments of the continuous data assimilation with the viscosity coefficient missing (VCM), and also show the results in Table 1 (see third and fourth lines). We find that, as time increases, the continuous data assimilation method successfully recovers the viscosity coefficient. Consequently, the numerical solutions obtained from the continuous data assimilation method become similar to those from the finite element approximation, even without the exact viscosity coefficient.
\begin{table}
    \centering
    \caption{Errors of the CDA on the rectangular domain}
     \begin{tabular}{cccc}
    \hline
    $t $& $||u_h^{\varepsilon}- v_h^{\varepsilon,\nu}  ||_{L^2(\Omega)}$(IDM) & $||u_h^{\varepsilon}- v_h^{\varepsilon,\nu}  ||_{L^2(\Omega)}$(VCM) & $| \nu - \tilde{\nu} |$ \\
    \hline
    0.1 & 3.644E+1 & 5.197E+1 & 4.976E-2 \\
    10  & 2.114E-1 & 2.738E-1 & 5.623E-4 \\
    20  & 4.613E-3 & 5.233E-3 & 7.516E-6 \\
    30  & 9.365E-6 & 2.132E-5 & 9.638E-8 \\
    \hline
    \end{tabular}

\end{table}

\subsubsection{Flow around a circular cylinder}
The flow around a circular cylinder is a classical benchmark problem in computational fluid dynamics. We now apply the continuous data assimilation with missing initial data or viscosity coefficients to this problem to further validate its effectiveness. The computational domain is set as $\Omega=(0,6)\times(0,1)$, which contains a circular cylinder centered at (1,0.5) with a radius $r=0.15$. The boundary conditions are defined as follows:
\begin{itemize}
    \item The inlet boundary condition on the $y$-axis is set by $(u_1,u_2)^\top=(-y(y-1),0)^\top$.
    \item The outlet boundary on the $x$-axis satisfies the zero gradient condition for both velocity and pressure.
    \item The boundary conditions on the circular cylinder and the top wall are homogeneous Dirichlet conditions.
    \item The bottom wall boundary is a nonlinear slip boundary condition.
\end{itemize}
We set the initial conditions as $u_{10}=u_{20}=e^{-(x+y)}$ when they are unknown in continuous data assimilation. Under the same computational environment, we show the simulations in Figures 3 and 4, as well as Table 2. These results exhibit phenomena strikingly similar to those observed in Subsection 3.4.1. Specifically, despite the absence of initial data or viscosity coefficients, the continuous data assimilation method is able to produce solutions that closely match those obtained using the finite element method with complete initial conditions when the time is large enough. These results confirm the correctness and effectiveness of the continuous data assimilation in solving complex fluid dynamics problems, even when critical information is initially unavailable.
\begin{figure}
    \caption*{FEM \qquad\qquad\qquad\qquad\qquad\qquad \quad CDA}
    \centering
    \begin{minipage}{0.45\textwidth}
        \centering
        \includegraphics[width=\linewidth]{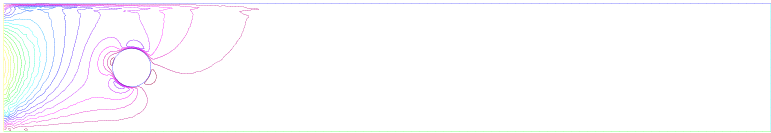}
    \end{minipage}%
    \begin{minipage}{0.45\textwidth}
        \centering
        \includegraphics[width=\linewidth]{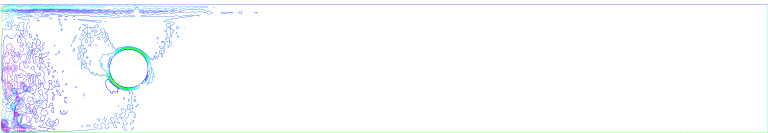}
    \end{minipage}%
    \vskip 0.2cm
    \caption*{$(a).~t=0.1$}

    \begin{minipage}{0.45\textwidth}
        \centering
        \includegraphics[width=\linewidth]{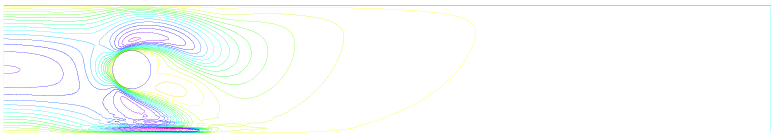}
    \end{minipage}%
    \begin{minipage}{0.45\textwidth}
      \centering
      \includegraphics[width=\linewidth]{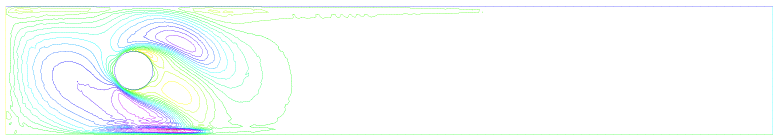}
    \end{minipage}%
    \vskip 0.2cm
    \caption*{$(b).~t=10$}

    \begin{minipage}{0.45\textwidth}
        \centering
        \includegraphics[width=\linewidth]{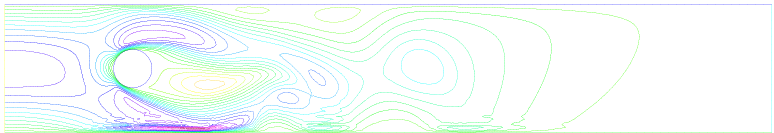}
    \end{minipage}
    \begin{minipage}{0.45\textwidth}
      \centering
      \includegraphics[width=\linewidth]{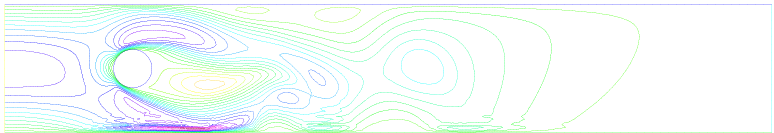}
    \end{minipage}
    \vskip 0.2cm
    \caption*{$(c).~t=30$}

    \caption{Comparisons of the simulations $u_{1h}^{\varepsilon}$ at different times (left: FEM, right: CDA).}

\end{figure}

\begin{figure}
    \centering
    \caption*{FEM \qquad\qquad\qquad\qquad\qquad\qquad \quad CDA}
    \begin{minipage}{0.45\textwidth}
      \centering
      \includegraphics[width=\linewidth]{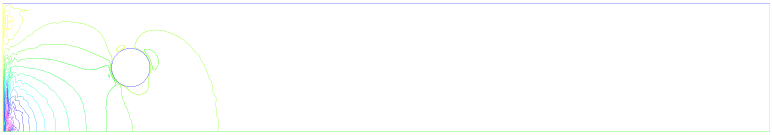}
    \end{minipage}
    \begin{minipage}{0.45\textwidth}
        \centering
        \includegraphics[width=\linewidth]{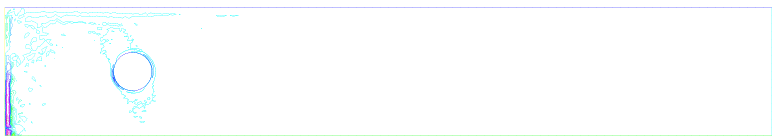}
    \end{minipage}%
    \vskip 0.2cm
    \caption*{$(a).~t=0.1$}

    \begin{minipage}{0.45\textwidth}
      \centering
      \includegraphics[width=\linewidth]{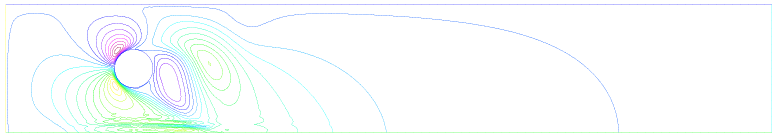}
    \end{minipage}%
    \begin{minipage}{0.45\textwidth}
        \centering
        \includegraphics[width=\linewidth]{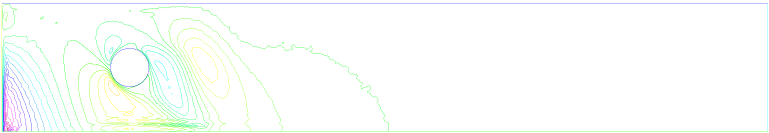}
    \end{minipage}
    \vskip 0.2cm
    \caption*{$(b).~t=20$}

    \begin{minipage}{0.45\textwidth}
      \centering
      \includegraphics[width=\linewidth]{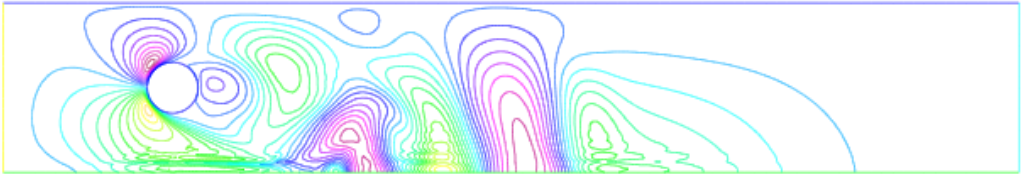}
  \end{minipage}
    \begin{minipage}{0.45\textwidth}
        \centering
        \includegraphics[width=\linewidth]{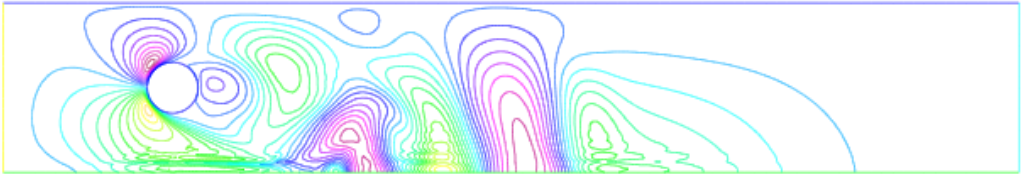}
    \end{minipage}
    \vskip 0.2cm
    \caption*{$(c).~t=30$}

    \caption{Comparisons of the simulations $u_{2h}^{\varepsilon}$ at different times (left: FEM, right: CDA).}

\end{figure}
\begin{table}
    \centering
   \caption{Errors of the CDA on the circular cylinder domain}
    \begin{tabular}{cccc}
    \hline
    $t $& $||u_h^{\varepsilon}- v_h^{\varepsilon,\nu}  ||_{L^2(\Omega)}$(IDM) & $||u_h^{\varepsilon}- v_h^{\varepsilon,\nu}  ||_{L^2(\Omega)}$(VCM) & $| \nu - \tilde{\nu} |$ \\
    \hline
    0.1 & 1.009E1   & 1.286E1   & 4.531E-2 \\
    10  & 1.042E-1  & 1.603E-1  & 1.359E-4 \\
    20  & 1.068E-3  & 1.129E-3  & 4.012E-6 \\
    30  & 5.217E-6  & 9.865E-5  & 1.112E-8 \\
    \hline
    \end{tabular}

\end{table}

\subsubsection{Bifurcated Blood Flow Model}
In this subsection, we further explore the capabilities of the continuous data assimilation method by applying it to a more complex fluid dynamics problem: a bifurcated blood flow model. The domain of this model is a ``Y"-shaped pipe (as depicted in Figure 5), consisting of a main vessel with a width of 2 and two branches, each with a width of 1. Thus, the domain examined in this case is significantly more complicated than those in the preceding examples. The boundary conditions in this model are as follows:
\begin{itemize}
    \item The inflow boundary condition is set on the right boundaries of the two branches, and the inflow velocity is defined as $ u_x = 1.2-1.2(y-1)^2,\: u_y = 0$.
    \item The outlet boundary on the $x$-axis satisfies the zero gradient condition for both velocity and pressure.
    \item The nonlinear slip boundary condition is applied to the bottom boundary of the main vessel.
    \item The homogeneous Dirichlet boundary condition is imposed on all other boundaries.
\end{itemize}
\begin{figure}
    \centering
    \includegraphics[width=0.3\textwidth]{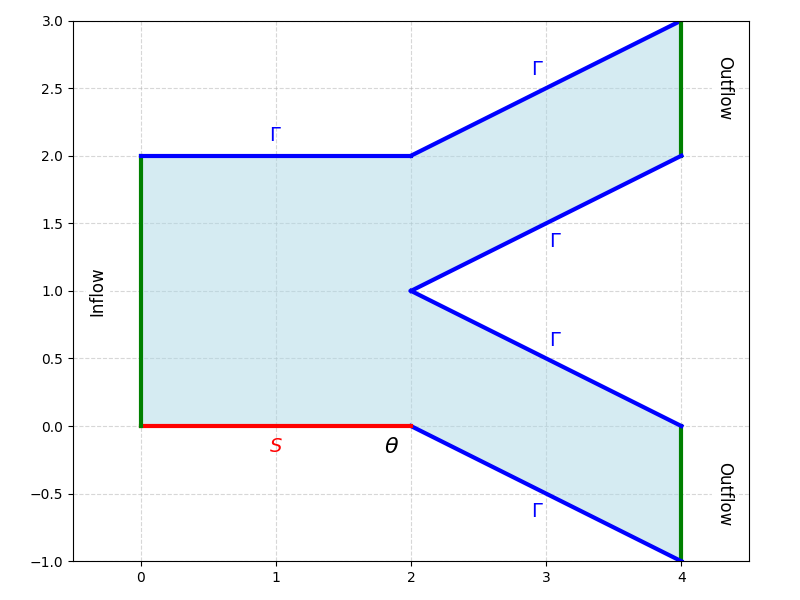}
    \caption{Domain and boundary conditions.}
  \end{figure}

We also set the initial conditions as $u_{10}=u_{20}=e^{-(x+y)}$ when they are unknown in continuous data assimilation. Using the same computational parameters, we conduct the continuous data assimilation and the simulation results are presented in Figures 6 and 7, as well as summarized in Table 3. These results are similar to the observations made in the previous sections, too. Even the initial data or viscosity coefficient missing, the continuous data assimilation method is capable of producing solutions that closely align with those obtained using the finite element method with complete initial conditions. As time progresses, the velocity fields from both methods are nearly indistinguishable.  This capability is particularly valuable in practical applications where complete data may not always be available.
\begin{figure}
    \centering
    \caption*{FEM \qquad\qquad\qquad\qquad\qquad\qquad \quad CDA}
    \begin{minipage}{0.3\textwidth}
        \centering
        \includegraphics[width=\linewidth]{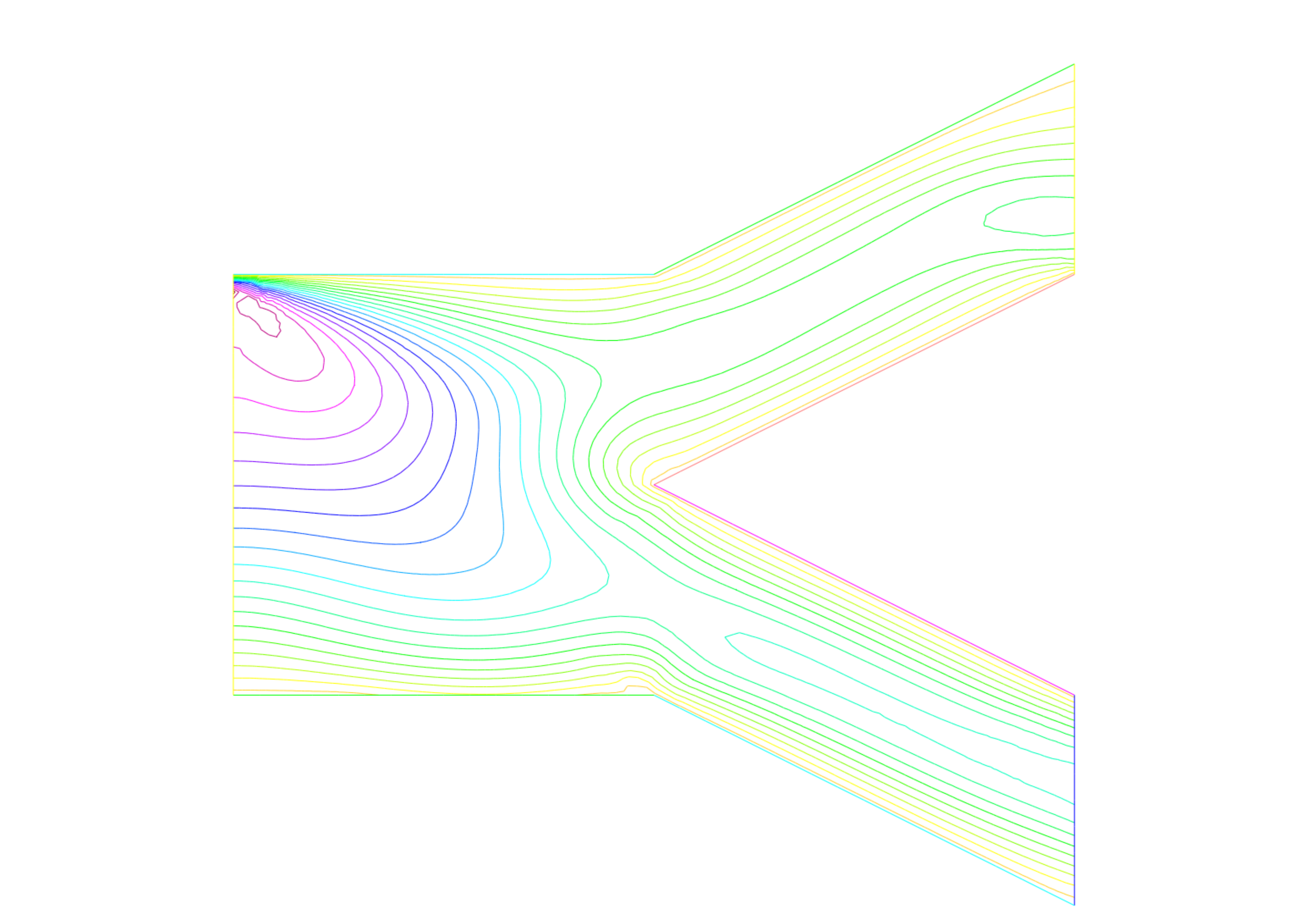}
    \end{minipage}%
    \begin{minipage}{0.3\textwidth}
      \centering
      \includegraphics[width=\linewidth]{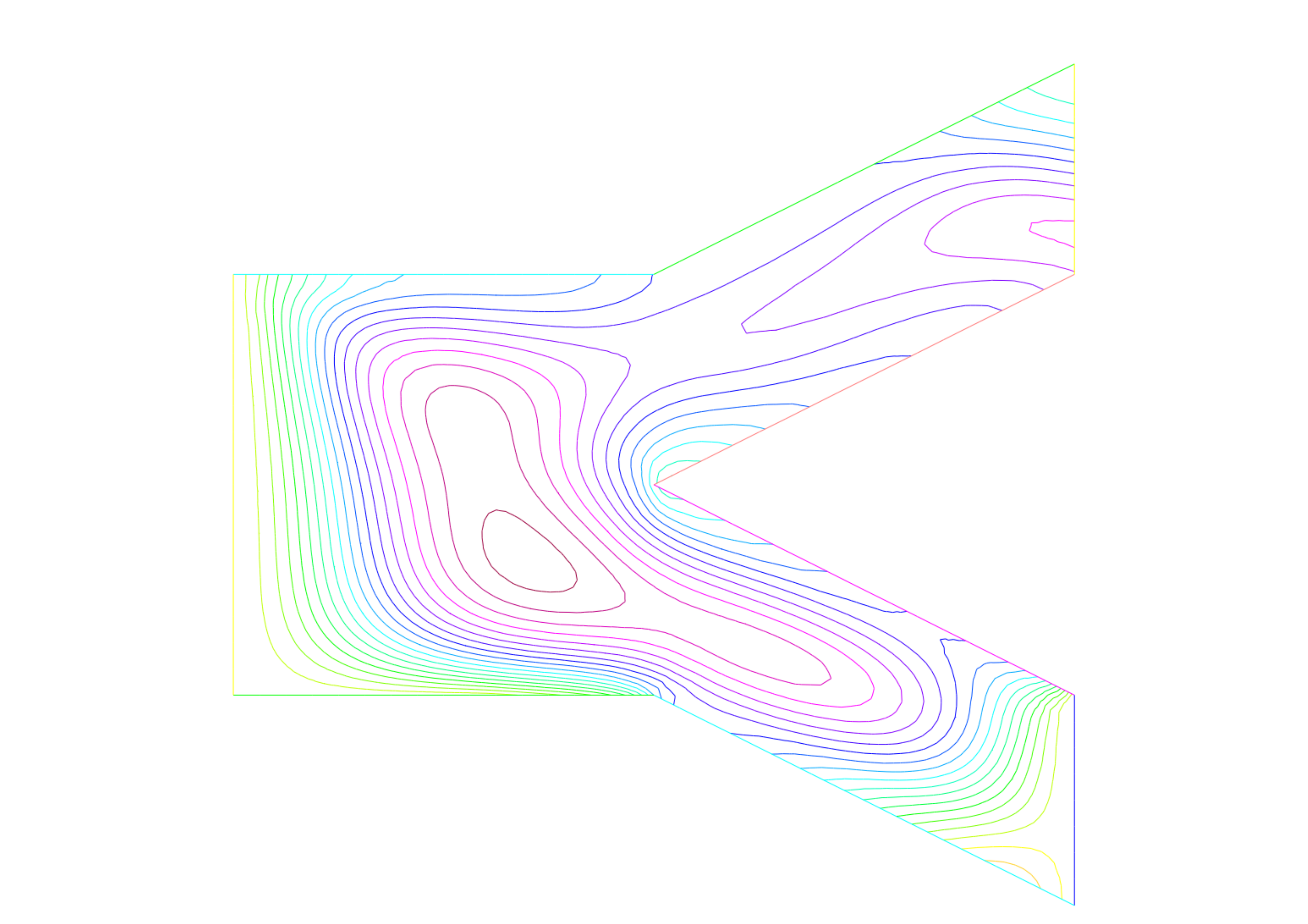}
    \end{minipage}%
    \vskip 0.2cm
    \caption*{$(a).~t=0.1$}

    \begin{minipage}{0.3\textwidth}
        \centering
        \includegraphics[width=\linewidth]{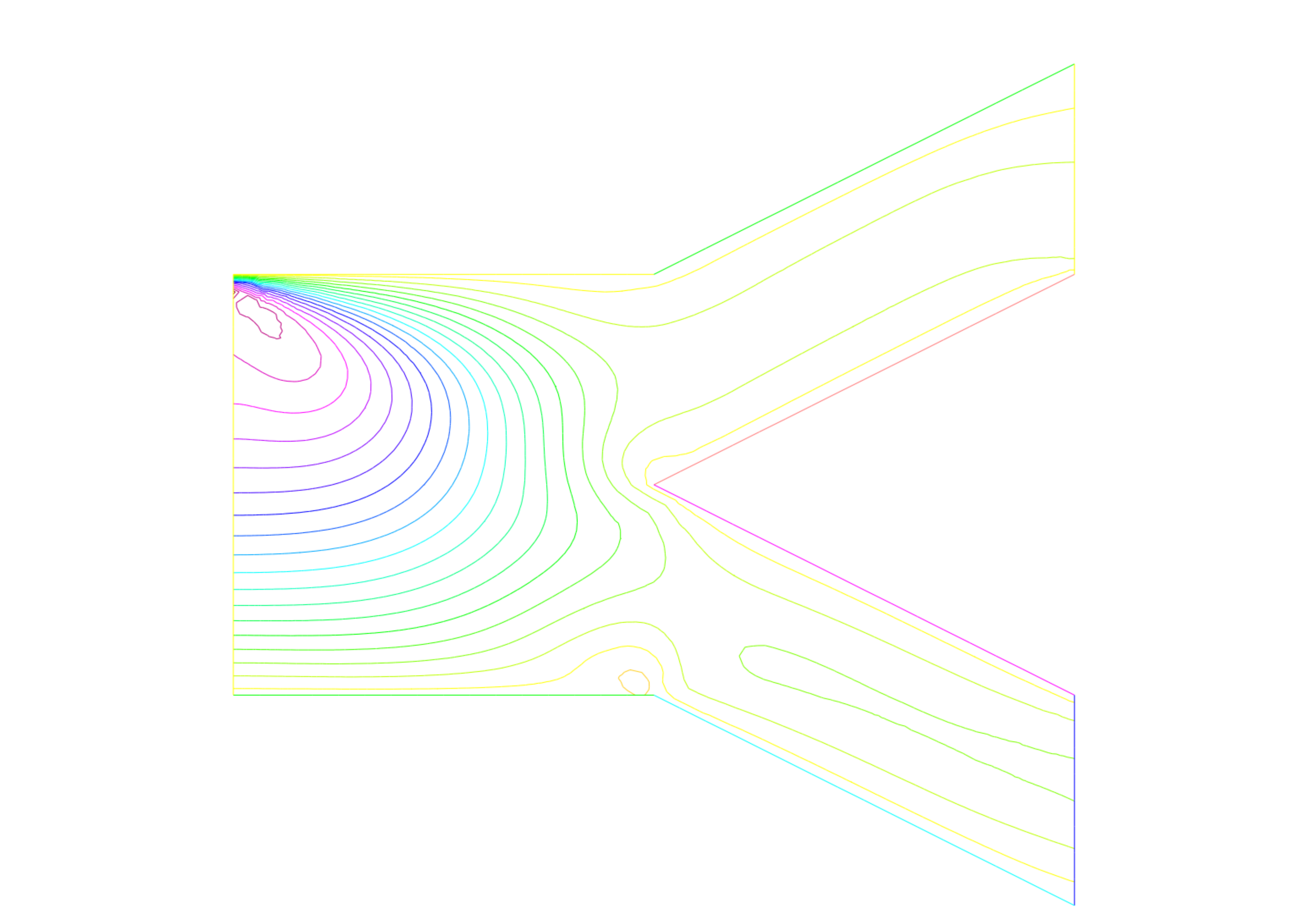}
    \end{minipage}%
    \begin{minipage}{0.3\textwidth}
      \centering
      \includegraphics[width=\linewidth]{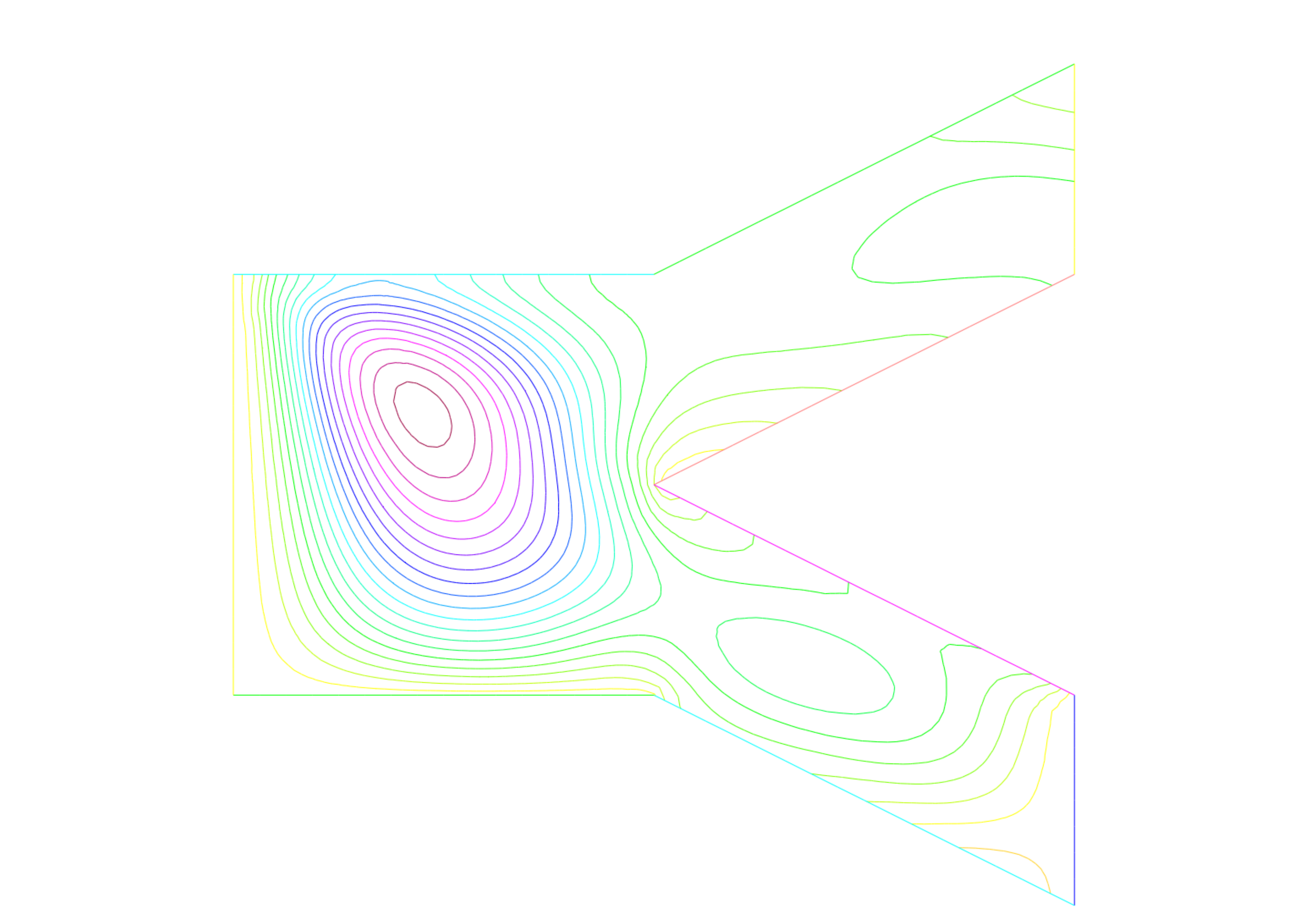}
    \end{minipage}%
    \vskip 0.2cm
    \caption*{$(b).~t=10$}

    \begin{minipage}{0.3\textwidth}
        \centering
        \includegraphics[width=\linewidth]{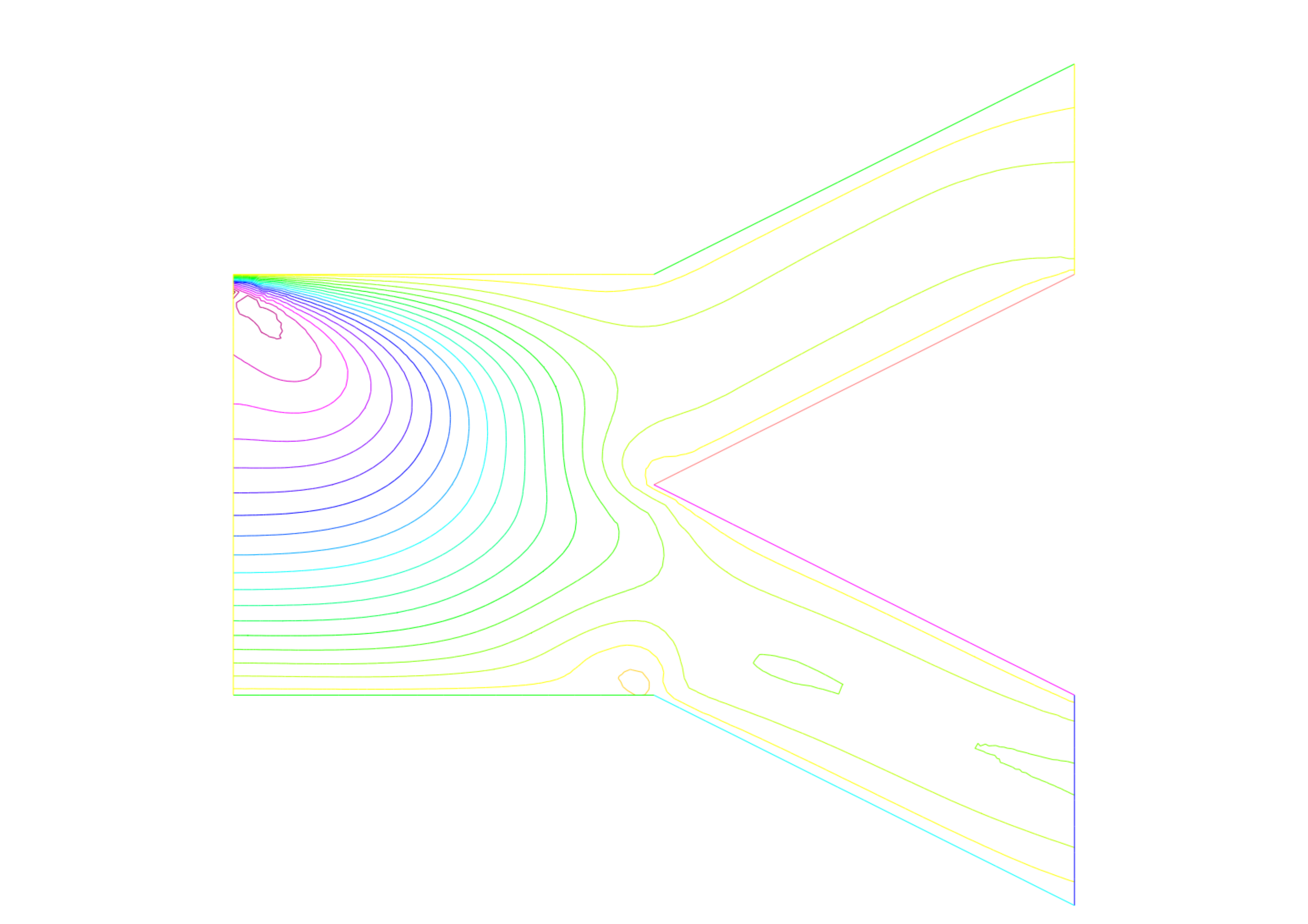}
    \end{minipage}%
    \begin{minipage}{0.3\textwidth}
      \centering
      \includegraphics[width=\linewidth]{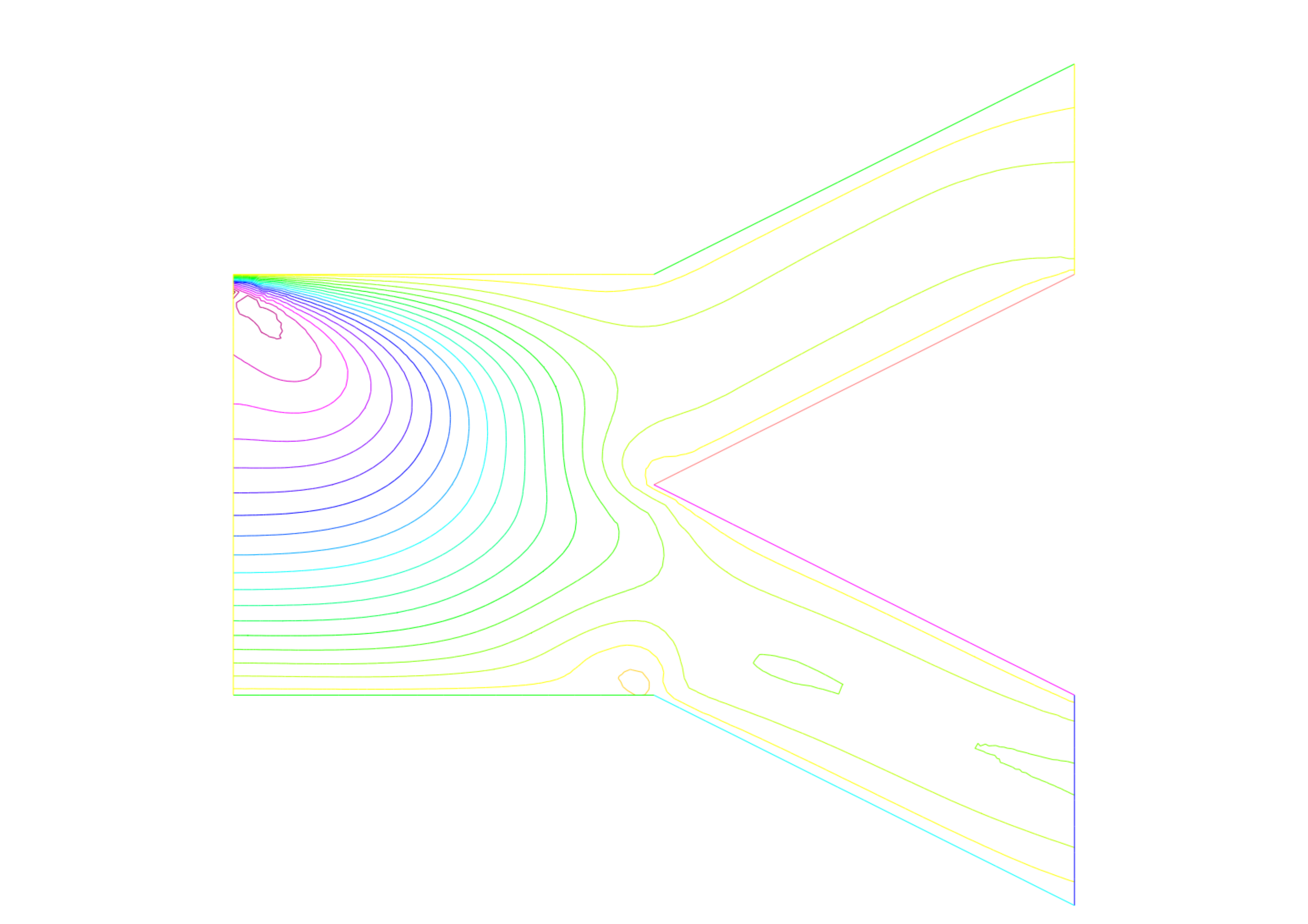}
    \end{minipage}%
    \vskip 0.2cm
    \caption*{$(c).~t=20$}

    \caption{Comparisons of the simulations $u_{1h}^{\varepsilon}$ at different times (left: FEM, right: CDA).}

\end{figure}

\begin{figure}[H]
    \centering
    \caption*{FEM \qquad\qquad\qquad\qquad\qquad\qquad \quad CDA}
    \begin{minipage}{0.3\textwidth}
      \centering
      \includegraphics[width=\linewidth]{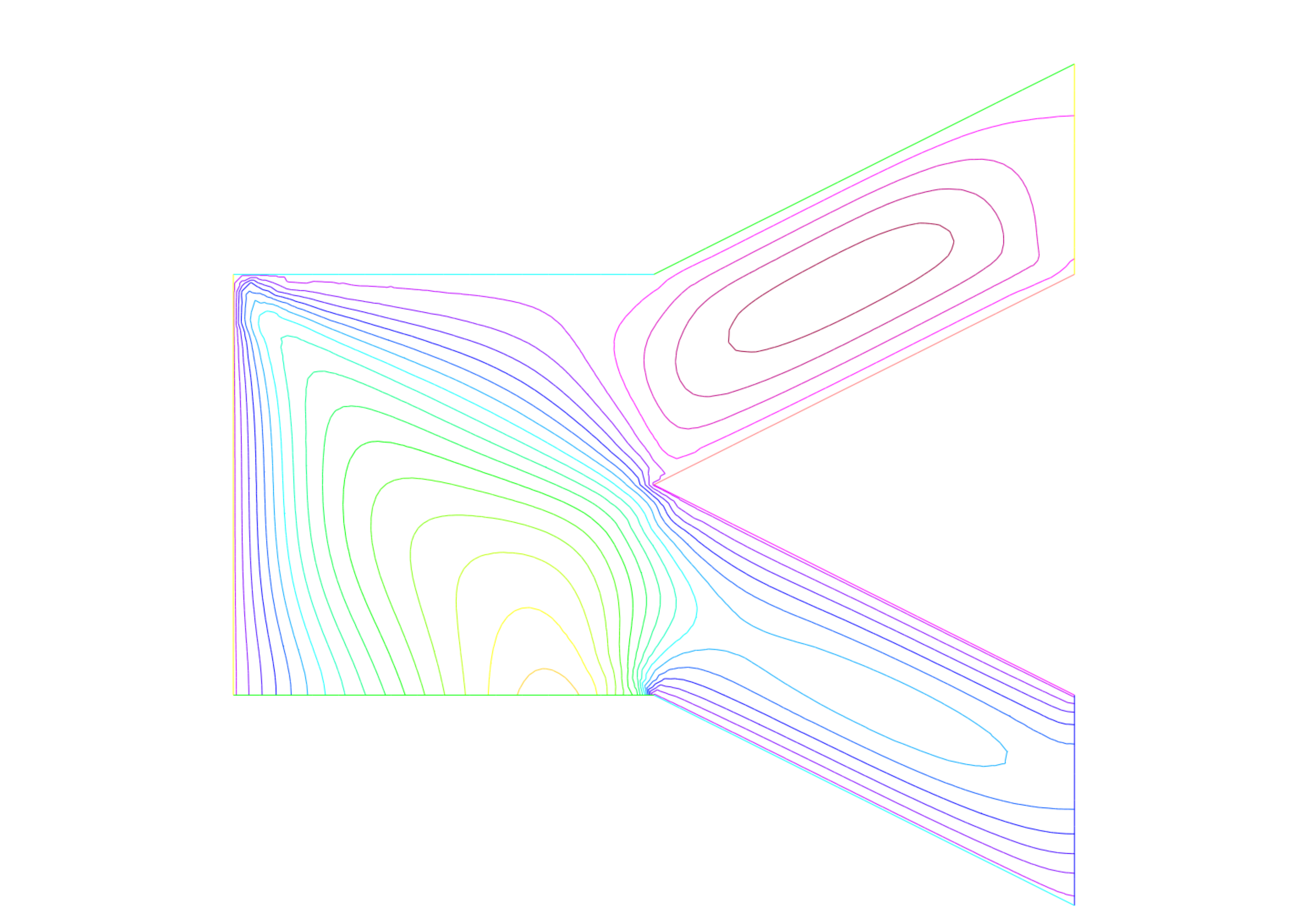}
    \end{minipage}
    \begin{minipage}{0.3\textwidth}
      \centering
      \includegraphics[width=\linewidth]{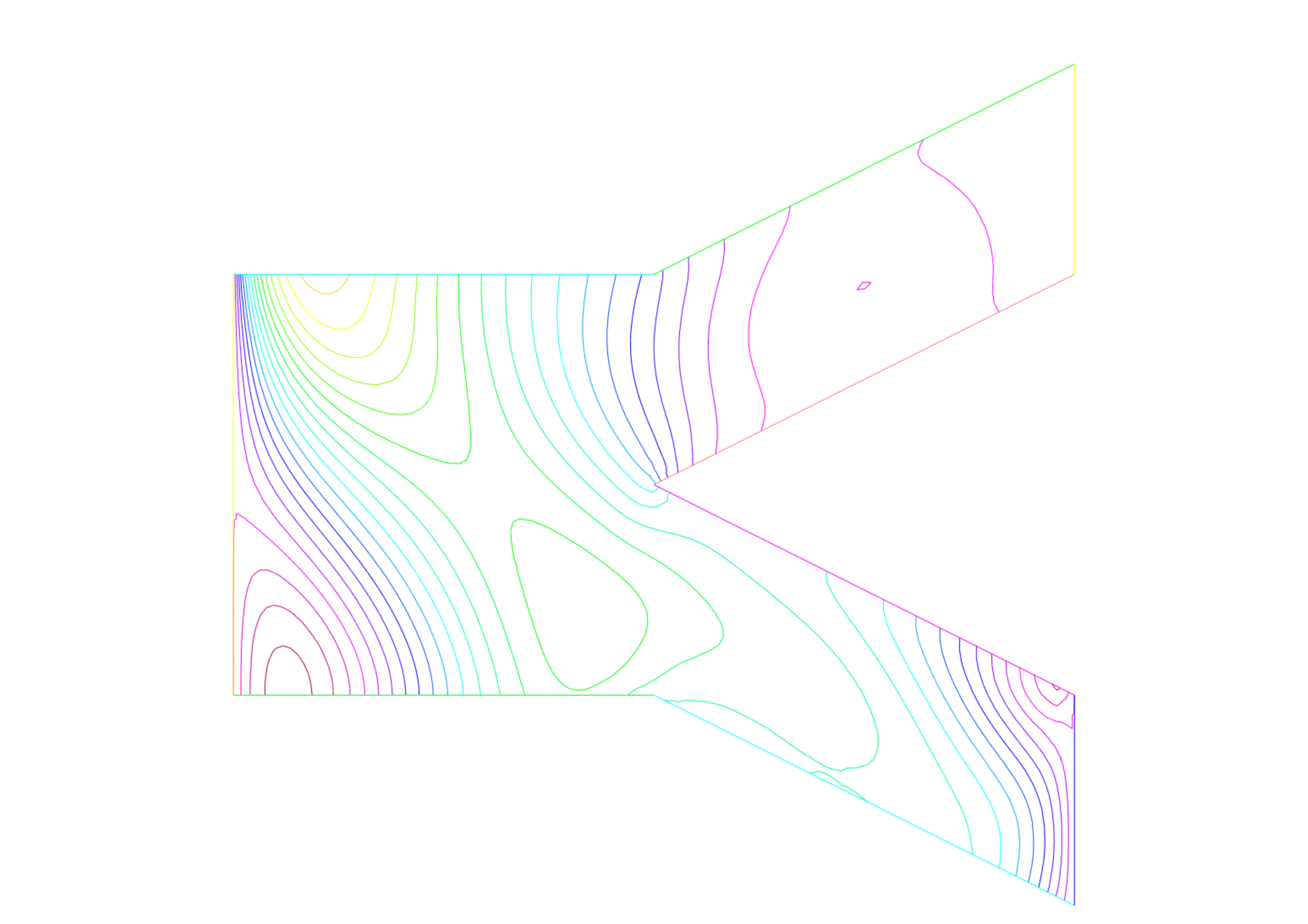}
    \end{minipage}
    \vskip 0.2cm
    \caption*{$(a).~t=0.1$}

    \begin{minipage}{0.3\textwidth}
      \centering
      \includegraphics[width=\linewidth]{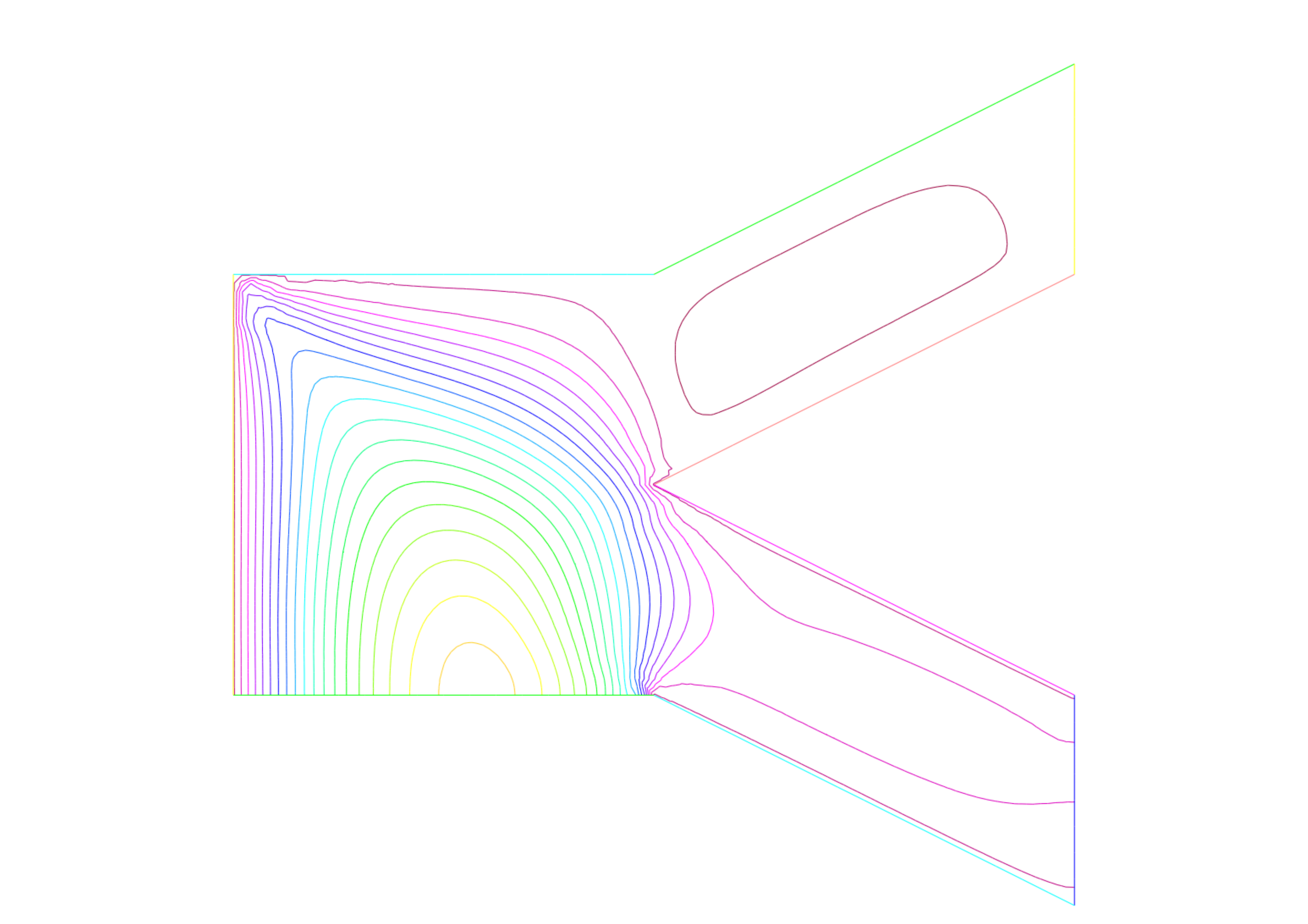}
    \end{minipage}
    \begin{minipage}{0.3\textwidth}
        \centering
        \includegraphics[width=\linewidth]{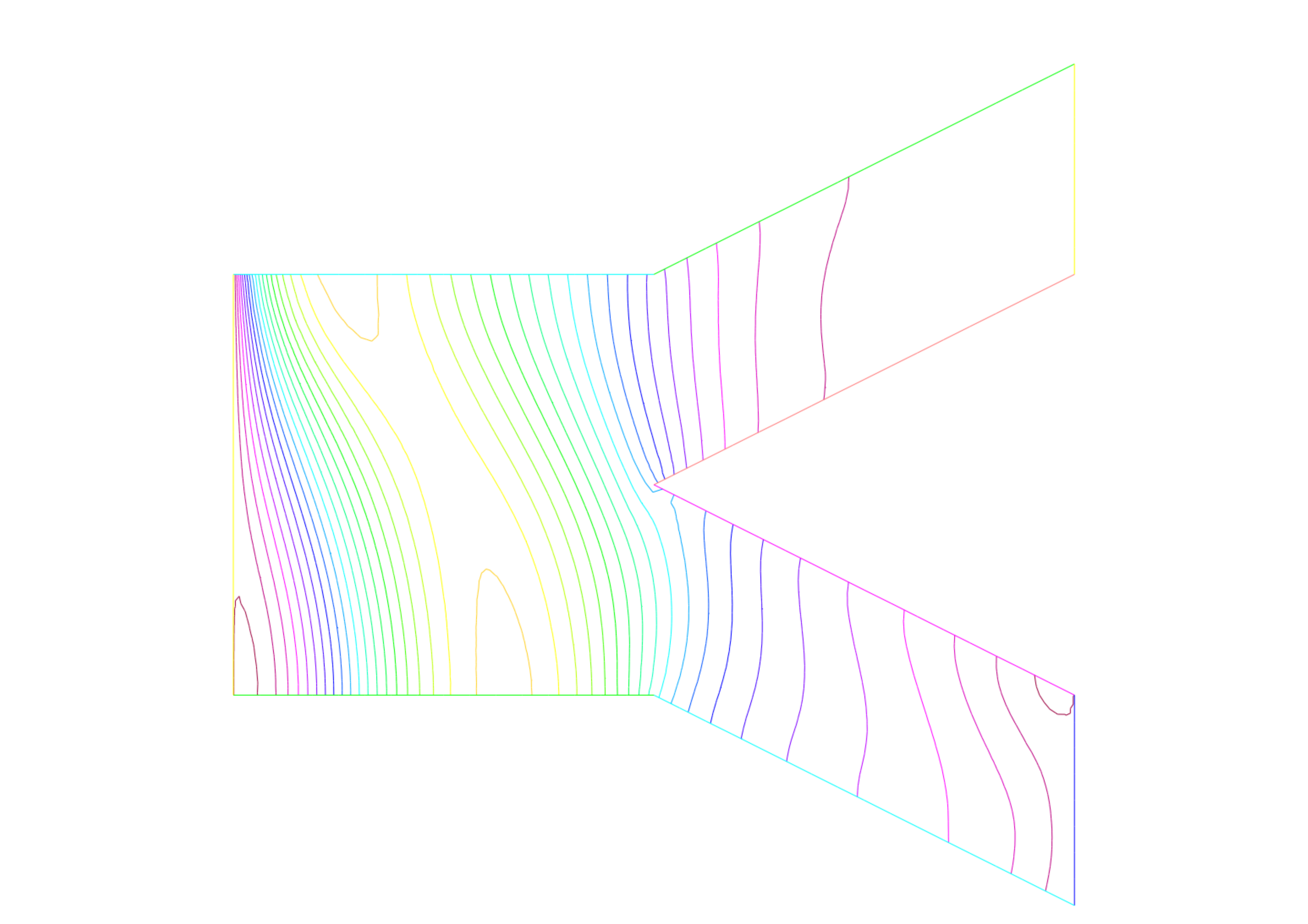}
    \end{minipage}
    \vskip 0.2cm
    \caption*{$(b).~t=10$}

    \begin{minipage}{0.3\textwidth}
      \centering
      \includegraphics[width=\linewidth]{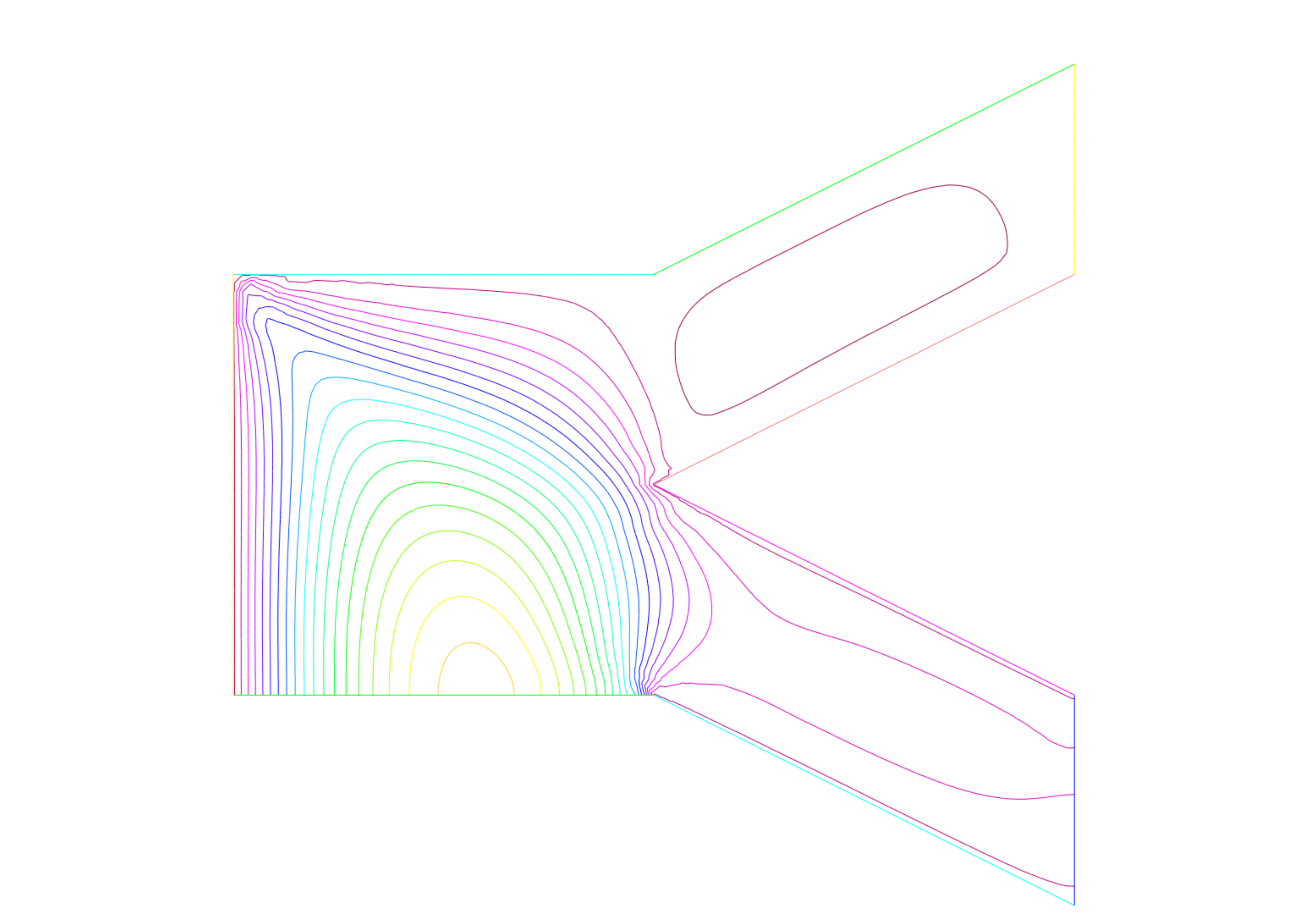}
    \end{minipage}%
    \begin{minipage}{0.3\textwidth}
        \centering
        \includegraphics[width=\linewidth]{figuer/true/U2Y-t=20.pdf}
    \end{minipage}%
    \vskip 0.2cm
    \caption*{$(c).~t=20$}

    \caption{Comparisons of the simulations $u_{2h}^{\varepsilon}$ at different times (left: FEM, right: CDA).}
\end{figure}

\begin{table}[H]
    \centering
      \caption{Errors of the CDA on the bifurcated domain}
      \begin{tabular}{cccc}
    \hline
    $t $& $||u_h^{\varepsilon}- v_h^{\varepsilon,\nu}  ||_{L^2(\Omega)}$(IDM) & $||u_h^{\varepsilon}- v_h^{\varepsilon,\nu}  ||_{L^2(\Omega)}$(VCM) & $| \nu - \tilde{\nu} |$ \\
    \hline
    0.1 & 9.597E-1 & 0.251E1  & 1.487E-2 \\
    10  & 8.339E-3 & 8.537E-1 & 5.007E-5 \\
    20  & 3.895E-5 & 4.217E-5 & 1.017E-7 \\
    \hline
    \end{tabular}
\end{table}

\section{Continuous Data Assimilation based on pETNNs}
When applying the classical CDA, observational data are required, which maybe challenge due to the limitation of the physical experiment environment.  Moreover, the simulation of the classical CDA  begins at the time obtained the observational data, that leads expensive computational cost in the long time approximation. In this section, utilizing the predictive capabilities of partial evolutionary tensor neural networks (pETNNs) for time-dependent problems, we study a novel CDA by replacing observational data with predictions generated by pETNNs, which combines the strengths of CDA and pETNNs and provide high-precision approximations regardless of observational data  for the Navier-Stokes equations with nonlinear slip boundary conditions. Next, we first recall  pETNNs in Subsection 4.1. Then, we will present the CDA based on pETNNs and show some numerical experiments to validate the efficiency of the new method in Subsection 4.2.

\subsection{Partial Evolutionary Tensor Neural Networks}
We begin by revisiting how a scalar function can be approximated using Tensor Neural Networks (TNNs). For more details, the reader is referred to, e.g., \cite{refn7,WXJ2024}. Consider a function \( \Phi(x) \), which is approximated by a TNN in the following product form:
\[
   \Phi(x;\theta)
   =
   \sum_{j=1}^{p}
   \Phi_{1,j}(x_1;\theta_1)\,
   \Phi_{2,j}(x_2;\theta_2)
   \cdots
   \Phi_{d,j}(x_d;\theta_d),
\]
where \(\theta = \{\theta_1,\dots,\theta_d\}\) encompasses all parameters of the architecture. One notable feature of TNNs is their compatibility with standard Gaussian quadrature for numerical integration, which can deliver improved accuracy (compared to Monte Carlo approaches) when incorporated into PDE-oriented loss functions via mean squared error formulations.\par

Noting that the TNNs was designed for steady-state problems, the authors in \cite{ref14} introduced \emph{partial evolutionary tensor neural networks} (pETNNs) to address time-dependent PDEs. In order to clarify the central idea, let us briefly outline pETNNs here (see \cite{ref14}). Consider a general time-dependent PDE of the form
\begin{equation}\label{eq:pde}
\begin{aligned}
   & \Phi_t - \mathcal{N} (\Phi) = 0, \quad (x,t)\in \Omega\times [0,T],\\
   & \Phi(x,0) = \Phi_0(x),
\end{aligned}
\end{equation}
where \(\Omega \subset \mathbb{R}^d\) is the spatial region and $d$ is spatial dimensions, \(T\) is the final time, \(\Phi(x,t)\) is the solution, and \(\mathcal{N}\) represents a nonlinear differential operator. Dividing the parameters into two parts, constants and time-varying variables, i.e., \(\theta = \{\hat{\theta}(t), \tilde{\theta}\}\), we get the pETNNs representation as
\begin{equation}\label{47}
   \Phi(x,t) \;\approx\;
   \hat{\Phi}\bigl(x,\theta(t)\bigr)
   \;=\;
   \sum_{j=1}^{p}
   \prod_{i=1}^d
   \hat{\Phi}_{i,j}\Bigl(x_i,\bigl(\hat{\theta}_i(t),\tilde{\theta}_i\bigr)\Bigr).
\end{equation}
In this setup, \(\hat{\theta}_i(t)\) are the time-dependent parameters for the \(i\)-th sub-network, while \(\tilde{\theta}_i\) remain constant.\par

Taking the time derivative of \(\hat{\Phi}\) in \eqref{47}, we obtain:
\[
   \frac{\partial \hat{\Phi}}{\partial t}
   \;=\;
   \frac{\partial \hat{\Phi}}{\partial \hat{\theta}}
   \,\frac{\partial \hat{\theta}}{\partial t}.
\]
Substituting this into the PDE \eqref{eq:pde} transforms the problem into the minimization
\begin{equation}\label{48}
\begin{aligned}
   & \frac{\partial \hat{\theta}}{\partial t}
     \;=\;
     \mathrm{argmin}\,J(\zeta), \\
   & \text{where }
     J(\zeta)
     \;=\;
     \frac{1}{2}
     \int_{\Omega}
       \Bigl\lvert
       \frac{\partial \hat{\Phi}}{\partial \hat{\theta}}\zeta
       \;-\;
       \mathcal{N}\bigl(\hat{\Phi}\bigr)
       \Bigr\rvert^2
     \,\mathrm{d}\textbf{x}.
\end{aligned}
\end{equation}
Solving this least squares problem yields an optimal \(\zeta_{\mathrm{opt}}\), interpretable as an approximation to \(\partial\hat{\theta}/\partial t\). Specifically,
\[
   \int_{\Omega}
     \Bigl(
       \frac{\partial \hat{\Phi}}{\partial \hat{\theta}}
     \Bigr)^{\!\top}
     \!\!\frac{\partial \hat{\Phi}}{\partial \hat{\theta}}
     \,\mathrm{d}\textbf{x}
   \,\zeta_{\mathrm{opt}}
   \;=\;
   \int_{\Omega}
     \Bigl(
       \frac{\partial \hat{\Phi}}{\partial \hat{\theta}}
     \Bigr)^{\!\top}
     \!\!\mathcal{N}\bigl(\hat{\Phi}\bigr)
     \,\mathrm{d}\textbf{x}.
\]
The \((i,j)\)-th entry of the matrix on the left is
\[
   \Biggl(
     \int_{\Omega}
       \Bigl(
         \frac{\partial \hat{\Phi}}{\partial \hat{\theta}}
       \Bigr)^{\!\top}
       \!\!\frac{\partial \hat{\Phi}}{\partial \hat{\theta}}
       \,\mathrm{d}\textbf{x}
   \Biggr)_{\!ij}
   =
   \int_{\Omega}
     \frac{\partial \hat{\Phi}}{\partial w_i}\,
     \frac{\partial \hat{\Phi}}{\partial w_j}
   \,\mathrm{d}\textbf{x},
\]
where each \(w_i\in \hat{\theta}(t)\) is a parameter in the neural network. Substituting the product form of \(\hat{\Phi}\) into the partial derivatives, we have
\[
   \frac{\partial \hat{\Phi}}{\partial w_i}
   \;=\;
   \frac{\partial }{\partial w_i}
   \Biggl(
   \sum_{j=1}^{p}
   \prod_{k=1}^d
   \hat{\Phi}_{k,j}\bigl(x_k,\hat{\theta}_k(t),\tilde{\theta}_k\bigr)
   \Biggr)
   =
   \sum_{j=1}^{p}
   \sum_{k=1}^d
     \frac{\partial \hat{\Phi}_{k,j}}{\partial w_k}
     \prod_{\substack{m=1 \\ m\neq k}}^d
       \hat{\Phi}_{m,j}\bigl(x_m,\hat{\theta}_m(t),\tilde{\theta}_m\bigr).
\]
Once \(\zeta_{\mathrm{opt}}\) is found by least squares, the time-dependent parameter \(\hat{\theta}(t)\) is updated via a predictor-corrector scheme. Denoting \(\hat{\theta}_n\) as the parameters at the \(n\)-th time step, we have:\par
\noindent \textbf{Prediction:}
\[
   \hat{\theta}^{p}_{n+1}
   =
   \hat{\theta}_n
   \;+\;
   \Delta t\,\zeta_{\mathrm{opt}}\bigl(\hat{\theta}_n\bigr),
\]
\noindent \textbf{Correction:}
\[
   \hat{\theta}_{n+1}
   =
   \hat{\theta}_n
   \;+\;
   \tfrac{\Delta t}{2}
   \Bigl(
   \zeta_{\mathrm{opt}}\bigl(\hat{\theta}_n\bigr)
   \;+\;
   \zeta_{\mathrm{opt}}\bigl(\hat{\theta}^{p}_{n+1}\bigr)
   \Bigr).
\]
Meanwhile, the parameters \(\tilde{\theta}\) of the static part  is optimized primarily through the initial/boundary data. For example, one may train the network so that
\[
   \bigl\lVert
     \hat{\Phi}(x,\theta_0)-\Phi_0(x)
   \bigr\rVert
\]
is minimized, ensuring that the TNN approximation aligns with the given initial condition \(\Phi_0(x)\).\par

\subsection{Continuous Data Assimilation based on pETNNs}
 Replacing the observational data obtained by the physical experiment with the predictions generated by the pETNNs (denoted as \( \hat{u}^{\varepsilon} \)),  we arrive at finite element approximation of the CDA based on pETNNs as follows: find $(\hat{v}_h^{\varepsilon}, \hat{p}_h^{\varepsilon})\in V_h\times M_h$ such that
\begin{equation}\label{n452}
\left\{\begin{aligned}
    &(\hat{v}_{ht}^{\varepsilon},w_h) + \nu a({ \hat{v}_h^{\varepsilon}  },w_h) + b({ \hat{v}_h^{\varepsilon}  },{ \hat{v}_h^{\varepsilon}  },w_h) - d(w_h,\hat{p}_h^{\varepsilon}) +  (\beta({ \hat{v}_{h\tau}^{\varepsilon}  }),{w_{h\tau}})_S\\
    &\qquad+\mu(I_{\hat{h}}( \hat{v}_h^{\varepsilon}  )-I_{\hat{h}}(\hat{u}^{\varepsilon}),w_h)= (f,w_h), \quad \forall w_h\in V_h.\\
    &d({ \hat{v}_h^{\varepsilon}  },q_h)= 0,\quad \forall q_h\in M_h.
\end{aligned}\right.
\end{equation}

Due to the strong time extrapolation capability of pETNNs, it is possible to obtain relatively reliable predictive solutions over extended time periods beyond the training interval. But the problem is that pETNNs require precise initial conditions to get accurate prediction results, however we only have the initial observations on the coarse grid. Therefore, despite the powerful time extrapolation ability of pETNNs, the high-precision model trained at the initial moment is lost due to the inaccurate initial information. Fortunately, the initial value selection in the CDA method is arbitrary, allowing us to utilize interpolated pETNNs predictions as the initial conditions for the CDA model near the time of interest, even if there is a certain amount of neural network calculation error. By employing a feedback control mechanism, the errors introduced by the neural network computations can be corrected within a short time frame.

\begin{table}
    \centering
    \caption{Comparisons of different methods}

    \begin{tabular}{cc}
        \begin{subtable}[t]{0.5\textwidth}
            \centering
            \begin{tabular}{cccc}
                \toprule
                $t$ & $\mathrm{CDApETNNs}$ & $\mathrm{pETNNs}$ & $\mathrm{CDA}$\\
                \midrule
                20 & 2.962E-3 & 2.962E-3 & 4.613E-3\\
                23 & 6.336E-4 & 3.091E-3 & 5.750E-4\\
                25 & 1.916E-4 & 3.031E-3 & 8.927E-5\\
                28 & 6.369E-5 & 2.834E-3 & 3.047E-5\\
                30 & 2.402E-5 & 3.020E-3 & 9.365E-6\\
                \bottomrule
            \end{tabular}
        \end{subtable}
    \end{tabular}
\end{table}

\subsection{Numerical experiment}

In this subsection, we utilize the same example as that in Section 3.4.1 to test the CDA based on pETNNs. Firstly, we train pETNNs by the same ``guessed" initial data at $t=0$ as that in the CDA since the initial data is absent. For the two sub-networks, we adopt $tanh$ activation and set two hidden layers with 30 neurons in each layer. Then we use this model to compute (\ref{23}). With the predictions generated by the pETNNs  on \([20,30]\)  as the "observational data", the CDA based on  pETNNs (CDApETNNs) is performed. For comparison, we also compute the errors by using the direct prediction from the pETNNs and the CDA. The results are presented in Table 4. From the table we can see that the CDApETNNs can achieve the similar approximation accuracy as the classical CDA, which is higher than that got by the pETNNs  alone. But, instead of the simulation beginning at $t=0$ in the classical CDA, the new proposed one needs much less computational time since it only need to simulate beginning at $t=20$.  The CDApETNNs combines the strengths of CDA and pETNNs and provide high-precision
approximations regardless of observational data.


%
%
%
\section{Conclusions}
In this paper, by employing a regularization method to handle the Navier-Stokes equations with nonlinear slip boundary conditions, we derive  the applicability of the CDA for this problem with the initial data and viscosity coefficient missing. Through various numerical experiments, the proposed method's stability and accuracy are validated. The results show that the proposed CDA framework achieves high-precision numerical solutions under low observation resolution and effectively recovers unknown parameters. Furthermore, the study integrates pETNNs to efficiently provide otherwise inaccessible observational values, which are successfully applied within the CDA framework. This idea can be extended to other  problems, such as the two-phase flow, the electrohydrodynamic and so on. We will consider these in future.


\begin{thebibliography}{99}

    \bibitem{aref7}Akbas M, Diegel A E, Rebholz L G. Continuous data assimilation of a discretized barotropic vorticity model of geophysical flow. Computers Mathematics with Applications, 2024, 160: 30-45.

    \bibitem{ref3}Azouani A, Olson E, Titi E S. Continuous data assimilation using general interpolant observables. Journal of Nonlinear Science, 2014, 24: 277-304.

    \bibitem{ref6}Barrat J L, Bocquet L. Large slip effect at a nonwetting fluid-solid interface. Physical Review Letters, 1999, 82(23): 4671.

    \bibitem{ref19}Beesack P R. On some Gronwall-type integral inequalities in $n$ independent variables. Journal of Mathematical Analysis and Applications, 1984, 100(2): 393-408.

    \bibitem{ref4}Biswas A, Price R. Continuous data assimilation for the three-dimensional Navier-Stokes equations. SIAM Journal on Mathematical Analysis, 2021, 53(6): 6697-6723.

    \bibitem{ref15}Carlson E, Hudson J, Larios A. Parameter recovery for the 2 dimensional Navier-Stokes equations via continuous data assimilation. SIAM Journal on Scientific Computing, 2020, 42(1): A250-A270.

    \bibitem{aref9}Carlson E, Farhat A, Martinez V R, et al. On the infinite-nudging limit of the nudging filter for continuous data assimilation. arXiv preprint arXiv:2408.02646, 2024.

    \bibitem{aref11}Carlson E, Larios A, Titi E S. Super-exponential convergence rate of a nonlinear continuous data assimilation algorithm: The 2D Navier-Stokes equation paradigm. Journal of Nonlinear Science, 2024, 34(2): 37.

    \bibitem{D2017}Djoko J K. Convergence analysis of the nonconforming finite element discretization of Stokes and Navier-Stokes equations with nonlinear slip boundary conditions. Numer. Funct. Anal. Optim., 2017, 38(8): 951-987.

    \bibitem{aref8}Diegel A E, Li X, Rebholz L G. Analysis of continuous data assimilation with large (or even infinite) nudging parameters. Journal of Computational and Applied Mathematics, 2025, 456: 116221.

    \bibitem{ref5}Gardner M, Larios A, Rebholz L G, et al. Continuous data assimilation applied to a velocity-vorticity formulation of the 2D Navier-Stokes equations. Electronic Research Archive, 2021, 29(3).

    \bibitem{ref2}Kalay E. Atmospheric modeling, data assimilation and predictability. Cambridge University Press, 2003.

    \bibitem{ref14}Kao T, Zhang H, Zhang L, et al. pETNNs: partial evolutionary tensor neural networks for solving time-dependent partial differential equations. arXiv preprint arXiv:2403.06084, 2024.

    \bibitem{aref10}Larios A, Victor C. Application of continuous data assimilation in high-resolution ocean modeling. Communications in Computational Physics, 2024, 35(5): 1418-1444.

    \bibitem{aref12}Larios A, Pei Y, Victor C. The second-best way to do sparse-in-time continuous data assimilation: Improving convergence rates for the 2D and 3D Navier-Stokes equations. arXiv preprint arXiv:2303.03495, 2023.

    \bibitem{ref7}Li Y, An R. Penalty finite element method for Navier-Stokes equations with nonlinear slip boundary conditions. International Journal for Numerical Methods in Fluids, 2012, 69(3): 550-566.

    \bibitem{ref8}Li Y, Li K. Pressure projection stabilized finite element method for Navier-Stokes equations with nonlinear slip boundary conditions. Computing, 2010, 87(3): 113-133.

    \bibitem{ref9}Li Y, An R. Semi-discrete stabilized finite element methods for Navier-Stokes equations with nonlinear slip boundary conditions based on regularization procedure. Numerische Mathematik, 2011, 117(1): 1-36.

    \bibitem{refn3}Liu L, Liu S, Xie H, et al. Discontinuity computing using physics-informed neural networks. Journal of Scientific Computing, 2024, 98(1): 22.

    \bibitem{ref17}Peng Y, Hu D, Xu Z Q J. A non-gradient method for solving elliptic partial differential equations with deep neural networks. Journal of Computational Physics, 2023, 472: 111690.

    \bibitem{AMZ2017}Qiu H, Mei L, Zhang Y. Two-grid variational multiscale algorithms for the stationary incompressible Navier-Stokes equations with friction boundary conditions. Numer. Methods Partial Differential Equations, 2017, 33(2): 546-569.

    \bibitem{ref10}Raissi M, Perdikaris P, Karniadakis G E. Physics-informed neural networks: A deep learning framework for solving forward and inverse problems involving nonlinear partial differential equations. Journal of Computational Physics, 2019, 378: 686-707.

    \bibitem{ref18}Roubinek T. Nonlinear partial differential equations with applications. Springer Science Business Media, 2013.

    \bibitem{insu}Schwab C. p-and hp-finite element methods: Theory and applications in solid and fluid mechanics. (No Title), 1998.

    \bibitem{refn1}Sarma A K, Roy S, Annavarapu C, et al. Interface PINNs (I-PINNs): A physics-informed neural networks framework for interface problems. Computer Methods in Applied Mechanics and Engineering, 2024, 429: 117135.

    \bibitem{ref16}Sirignano J, Spiliopoulos K. DGM: A deep learning algorithm for solving partial differential equations. Journal of Computational Physics, 2018, 375: 1339-1364.

    \bibitem{refn5}Song J, Cao W, Liao F, et al. VW-PINNs: A volume weighting method for PDE residuals in physics-informed neural networks. Acta Mechanica Sinica, 2025, 41(3): 324140.

    \bibitem{ref20}Temam R. Navier-Stokes equations: theory and numerical analysis. American Mathematical Society, 2001.

    \bibitem{ref21}Verferth R. Finite element approximation on incompressible Navier-Stokes equations with slip boundary condition. Numerische Mathematik, 1986, 50: 697-721.

    \bibitem{refn7}Wang Y, Lin Z, Liao Y, Liu H, Xie H. Solving high-dimensional partial differential equations using tensor neural network and a posteriori error estimators. Journal of Scientific Computing, 2024, 101(3): 1-29.

    \bibitem{WXJ2024}Wang Y, Xie H, Jin P. Tensor neural network and its numerical integration. Journal of Computational Mathematics, 2024, 42(6): 1714-1742.

    \bibitem{aref6}You B. Continuous data assimilation for the three-dimensional planetary geostrophic equations of large-scale ocean circulation. Zeitschrift faar angewandte Mathematik und Physik, 2024, 75(4): 147.

    \bibitem{refn2}Zou Z, Meng X, Karniadakis G E. Correcting model misspecification in physics-informed neural networks (PINNs). Journal of Computational Physics, 2024, 505: 112918.

    \bibitem{refn4}Zhang Z, Zou Z, Kuhl E, et al. Discovering a reaction-diffusion model for Alzheimer's disease by combining PINNs with symbolic regression. Computer Methods in Applied Mechanics and Engineering, 2024, 419: 116647.

    \bibitem{ZS2022}Zheng B, Shang Y. A three-step defect-correction algorithm for incompressible flows with friction boundary conditions. Numerical Algorithms, 2022, 91(4): 1483-1510.

    \bibitem{ZS2023}Zheng B, Shang Y. Parallel defect-correction methods for incompressible flows with friction boundary conditions. Computer \& Fluids, 2023, 251: 105733.

    \end{thebibliography}
\end{document}